\documentclass[11pt, leqno, twoside]{article}

\usepackage{amssymb}
\usepackage{amsmath}
\usepackage{amsthm}
\usepackage{amsfonts}
\usepackage{mathrsfs}
\usepackage{indentfirst}
\usepackage{graphicx}
\usepackage{txfonts}

\usepackage{anysize}

\usepackage{tikz}
\usetikzlibrary{positioning}

\usepackage{fancyhdr}
\usepackage{tikz}
\usetikzlibrary{positioning}

\usepackage{color}
\usepackage{setspace}

\allowdisplaybreaks

\usepackage{txfonts}

\def\rr{{\mathbb R}}
\def\rn{{{\rr}^n}}

\def\zz{{\mathbb Z}}

\def\nn{{\mathbb N}}

\def\cm{{\mathcal M}}

\def\ca{{\mathcal A}}

\def\ch{{\mathcal H}}

\def\fz{\infty}
\def\az{\alpha}

\def\dz{\delta}

\def\lz{\lambda}

\def\lf{\left}
\def\r{\right}

\def\ls{\lesssim}
\def\gs{\gtrsim}

\def\noz{\nonumber}
\def\wz{\widetilde}

\def\loc{{\mathop\mathrm{\,loc\,}}}
\def\esinf{\mathop\mathrm{\,ess\,inf\,}}
\def\esup{\mathop\mathrm{\,ess\,sup\,}}

\def\q1{\wz q}
\def\Q1{q_1}

\def\loc{{\mathop\mathrm{loc}}}

\newtheorem{thm}{Theorem}[section]
\newtheorem{prop}[thm]{Proposition}
\newtheorem{lem}[thm]{Lemma}
\newtheorem{cor}[thm]{Corollary}

\theoremstyle{definition}

\newtheorem{rem}[thm]{Remark}

\numberwithin{equation}{section}

\begin{document}
\title{\bf\Large Capacitary Muckenhoupt Weight, BMO and BLO Spaces with Hausdorff Content, Factorization Theorems and Applications
\footnotetext{\hspace{-0.35cm}
2020 {\it Mathematics Subject Classification}.
{Primary 42B25; Secondary 28A12, 28A25, 31C15.} \endgraf
{\it Key words and phrases.} Capacitary Muckenhoupt weight, BMO, BLO, Hausdorff content, factorization theorem, John--Nirenberg inequality, Calder\'on--Zygmund decomposition. \endgraf}}
\author{Long Huang, Yangzhi Zhang and Ciqiang Zhuo\footnote{Corresponding author, E-mail: cqzhuo87@hunnu.edu.cn
/{\color{red}{January 31, 2026}}
}
}
\date{ }
\maketitle
	
\vspace{-0.7cm}

\begin{center}
\begin{minipage}{13cm}
{\small {\bf Abstract}\quad
Let $\delta\in(0,n]$, $p\in[1,\infty)$, $\mathcal H_{\infty}^\delta$ denote the Hausdorff content on $\mathbb R^n$, and $\mathcal A_{p,\delta}$ be the capacitary Muckenhoupt weight class.
We are interested in understanding the relationship between the capacitary Muckenhoupt weight class
$\mathcal A_{p,\delta}$ and ${\rm{BMO}}(\mathbb R^n, \mathcal H_{\infty}^{\delta})$ or ${\rm{BLO}}(\mathbb R^n, \mathcal H_{\infty}^{\delta})$ spaces for all dimension $\delta\in(0,n]$,
and further to comprehend the structure of these two spaces. 
Our main result shows that $\mathcal A_{p,\delta}$ for $p\in(1,\infty)$ is equivalent to the BMO spaces, while $\mathcal A_{1,\delta}$ is equivalent to the BLO spaces, and consequently yields the factorization theorems for these BMO and BLO spaces via capacitary Hardy--Littlewood maximal operators, which essentially extend main results of Coifman and Rochberg in 1980 beyond measure theory. As applications,
by establishing some capacitary weighted John--Nirenberg inequalities, we obtain the  equivalence between capacitary weighted BMO or BLO spaces and 
${\rm{BMO}}(\mathbb R^n, \mathcal H_{\infty}^{\delta})$ or ${\rm{BLO}}(\mathbb R^n, \mathcal H_{\infty}^{\delta})$ respectively. These results reveal deep connections between $\mathcal A_{p,\delta}$ and BMO or BLO spaces with Hausdorff content, beyond the classical measure-theoretic settings. We develop some approaches in the proofs and using a new observation, that is, the additivity of measures and linearity of integrals are superfluous for the corresponding classical theory.}
\end{minipage}
\end{center}


\vspace{0.2cm}

\section{Introduction}

\overfullrule=5pt

Harmonic analysis on Euclidean spaces endowed with the Lebesgue measure, or with measures absolutely continuous with respect to it, has flourished since the last century. The theory has been well established and now serves as a foundation for many areas of mathematics; see, for instance, the book \cite{swbook} by Stein and Weiss. In recent years, we have been devoted to essentially extending measure-based harmonic analysis on the Euclidean space $\rn$ to the setting of $\rn$ equipped with low-dimensional outer measures, with the aim of understanding the structure of low-dimensional sets in $\rn$, including sets with fractional dimensions such as fractals. At the same time, we investigate which features used in the development of classical harmonic analysis—such as the additivity of measures and the linearity of integrals—are not intrinsic or indispensable. 
Roughly speaking, we employ a finer ``measuring scale", namely Hausdorff content, to quantify the size of sets in $\rn$, and seek to develop the corresponding theory of harmonic analysis. 

The Hausdorff content fails to be absolutely continuous with respect to Lebesgue measures and instead admits both singular and absolutely continuous parts.
Given a set $E\subset \rn$ and $\dz \in (0,n]$,
the \emph{Hausdorff content} $\mathcal{H}_{\infty }^{\delta }$ of $E$ is defined by setting
\begin{align*}
\mathcal{H}_{\infty }^{\delta }\lf (E \r )
\coloneqq \inf\lf \{\sum_i [l(Q_i)]^{\delta }:\
E\subset \bigcup_i Q_i \right \},
\end{align*}
where the infimum is taken over all finite or countable cubes coverings
$\{Q_i\}_{i}$ of $E$ and $l(Q)$ denotes the edge length
of the cube $Q\subset\rn$. Here and thereafter, a cube $Q\subset\rn$ always denotes a left-open and right-closed cube in $\rn$ with side parallel to the coordinate axes. The quantity $\mathcal{H}_{\infty }^{\delta }\lf (E \r )$ is also referred to as the $\delta$-\emph{Hausdorff capacity} or the \emph{Hausdorff content of $E$ of
dimension} $\delta$. The Hausdorff content is closely related to the Hausdorff measure, but perhaps less popular both in geometric measure theory and harmonic analysis. Nevertheless, it offers several advantages: it enjoys greater regularity than 
Hausdorff measure, easier to analyze and still gives the Hausdorff dimension as the critical exponent. For more relationships between Hausdorff content and Hausdorff measure, we refer to \cite{ff15}.

For any subset $E\subset \rn$ and non-negative function $f$, the \emph{Choquet integral} of $f$ on $E$ is defined by setting
\[\int_{E}f(x)\,d\ch_{\fz}^{\dz}\coloneqq\int_{0}^{\fz}\ch_{\fz}^{\dz}(\{x\in E:\ f(x)>t\})\,dt.\]
From the definition, it follows that the Choquet integral is well defined for all non-negative functions, without the assumption of measurability. For any $p\in(0,\fz)$, the \emph{$p$-Choquet
integral with respect to the capacity} of a function $f$ on $E$ is defined by setting
\begin{align*}
\|f\|_{L^p(E,\ch_{\fz}^{\dz})}
\coloneqq\lf[\int_E |f(x)|^p\,d\ch_{\fz}^{\dz}\r]^\frac1p.
\end{align*}
For any $p\in(0,\fz)$ and $E\subset \rn$, we always use $L^p(E,\ch_{\fz}^{\dz})$ to denote
the space of all functions $f$ such that the quasi-norm $\|f\|_{L^p(E,\ch_{\fz}^{\dz})}$ is finite, and denote by $L_{\loc}^1(\rn,\ch_{\fz}^{\dz})$ the set of all functions $f$ satisfying that $f\in L^1(K,\ch_{\fz}^{\dz})$ for any compact set $K\subset\rn$.

The capacitary Muckenhoupt weight class
$\ca_{p,\dz}$, depending not only on the integrable exponent $p$ but also on the dimension $\delta$ of Hausdorff content, was recently introduced and studied by Huang et al. in \cite{0702-4}.
Precisely, if $w\in L_{\loc}^1(\rn,\ch_{\fz}^{\dz})$ and $w\in(0, \fz)$ for $\ch_{\fz}^{\dz}$-almost everywhere, then we say that $w$ is a weight. For any $p\in (1,\fz)$ and $\delta\in(0,n]$, the capacitary Muckenhoupt weight class $\ca_{p,\dz}$ is defined by setting
\[\ca_{p,\dz}\coloneqq\lf\{w: [w]_{\mathcal A_{p,\delta}}\coloneqq\sup_{Q}\lf\{\frac{1}{\mathcal H^{\delta}_\infty(Q)}\int_Qw(x)\,d\mathcal H^{\delta}_\infty\r\}\lf\{\frac{1}{\mathcal H^{\delta}_\infty(Q)}\int_Qw(x)^{-\frac{1}{p-1}}\,d\mathcal H^{\delta}_\infty\r\}^{p-1}< \fz\r\}\]
with the supremum being taken over all cubes $Q\subset \rn$,
and
\[\ca_{1,\dz}\coloneqq\lf\{w: [w]_{\ca_{1,\delta}}\coloneqq\inf\lf\{c\in(0,\fz):\ \cm_{\ch_\fz^\delta}w(x)\leq cw (x)\ {\rm for}\ \ch_\fz^\delta{\rm-a.e.}~x\in\rn\r\}<\fz\r\}.\]
Here and thereafter, $\mathcal M_{\mathcal H_{\infty}^\delta}$ is the capacitary Hardy--Littlewood maximal operator associated with Hausdorff contents defined by
$$\mathcal M_{{\mathcal H}_\infty^\delta}f(x)\coloneqq\sup_{Q\ni x}\frac1{{\mathcal H}_\infty^\delta(Q)}\int_Q |f(y)|\,d{\mathcal H}_\infty^\delta$$
with the supremum taken over all cubes $Q\subset \rn$ containing $x$ and $f\in L_{\loc}^1(\rn,{\mathcal H^{\delta}_\infty})$.
This weight class $\ca_{p,\dz}$ revisits the  Muckenhoupt weight class $A_p$ when dimension $\delta=n$, and $\ca_{p,\dz}\subsetneqq A_p$ for dimension $\delta\in(0,n)$ and $p\in[1,\fz)$. The analogues of Muckenhoupt's theorem, reverse H\"older inequality, self-improving property and Jones' factorization theorem were built under Choquet integral setting in \cite{0702-4}. Consequently, the classical Muckenhoupt $A_p$-weight theory was established beyond measure-theortic frameworks.

We now summarize the main theorems of this paper in the diagrams below, and then proceed to present the details.

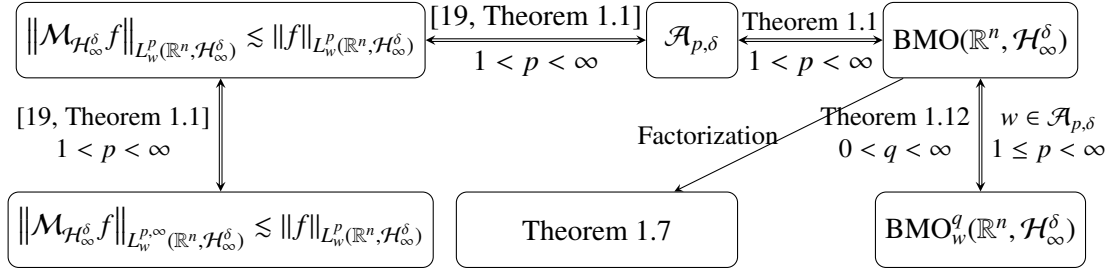
\begin{figure}[htbp]
\centering
\begin{tikzpicture}[
    node distance=1.8cm,
    box/.style={draw, rectangle, rounded corners, minimum height=1cm, minimum width=1.2cm, align=center},
    thm/.style={font=\small, midway, above}
]

\node[box] (Mf) {$ \lf\|\cm_{\ch_\fz^\dz}f\r\|_{L^{p}_{ w}(\rn, \ch_{\fz}^{\dz})}\ls \|f\|_{L^{p}_{ w}(\rn, \ch_{\fz}^{\dz})}$};
\node[box, right=2.9cm of Mf] (Aps) {$\mathcal A_{p,\dz}$};
\node[box, right=1.9cm of Aps] (BMO) {${\rm BMO}(\mathbb R^n, \mathcal H_{\infty}^{\dz})$};
\node[box, below=1.5cm of BMO] (BMOW) {${\rm BMO}_w^q(\mathbb R^n, \mathcal H_{\infty}^{\dz})$};
\node[box, left=of BMOW, minimum width=3.7cm] (ABCDE) {Theorem \ref{cor-0721-1}};
\node[box, below=1.5cm of Mf] (wMf) {$ \lf\|\cm_{\ch_\fz^\delta}f\r\|_{L^{p,\fz}_{ w}(\rn, \ch_{\fz}^{\dz})}\ls\|f\|_{L^{p}_{w}(\rn, \ch_{\fz}^{\dz})}$};

\draw[<->, double, scale=0.8, >=stealth] (Mf) -- (Aps) node[midway, above] {\cite[Theorem 1.1]{0702-4}}
node[midway, below] {$1<p<\fz$};
\draw[<->, double, scale=0.8, >=stealth] (Aps) -- (BMO) node[thm] {Theorem \ref{tm-0628-1}}
node[midway, below] {$1<p<\fz$};
\draw[<->, double, scale=0.8, >=stealth] (BMO) -- (BMOW) node[thm, left, align=center] {Theorem \ref{them-0629-1}\\
$0<q<\fz$}
node[thm, right, align=center] {$w\in \mathcal A_{p,\dz}$\\
$1\le p<\fz$};
\draw[<->, double, scale=0.8, >=stealth] (Mf) -- (wMf) 
node[thm, left, align=center] {\cite[Theorem 1.1]{0702-4}\\
$1< p<\fz$};
\draw[<-,  scale=0.8, >=stealth] (ABCDE) -- (BMO) node[thm, left] {Factorization};
\end{tikzpicture}
\caption{Equivalence of $\mathcal A_{p,\delta}$ and ${\rm BMO}(\mathbb R^n, \mathcal H_{\infty}^{\delta})$ for all dimension $\delta\in(0,n]$.}
\end{figure}

\begin{figure}[htbp]
\centering
\begin{tikzpicture}[
    node distance=1.8cm,
    box/.style={draw, rectangle, rounded corners, minimum height=1cm, minimum width=1.2cm, align=center},
    thm/.style={font=\small, midway, above}
]

\node[box] (Mf) {$ \lf\|\cm_{\ch_\fz^\delta}f\r\|_{L^{1,\fz}_{ w}(\rn, \ch_{\fz}^{\dz})}\ls\|f\|_{L^{1}_{w}(\rn, \ch_{\fz}^{\dz})}$};
\node[box, right=2.85cm of Mf] (Aps) {$\mathcal A_{1,\dz}$};
\node[box, right=1.83cm of Aps] (BMO) {${\rm BLO}(\mathbb R^n, \mathcal H_{\infty}^{\dz})$};
\node[box, below=1.4cm of BMO] (BMOW) {${\rm BLO}_w^q(\mathbb R^n, \mathcal H_{\infty}^{\dz})$};
\node[box, left=of BMOW, minimum width=4.5cm] (ABCDE) {Theorem \ref{cor-0721-2}};

\draw[<->, double, scale=0.8, >=stealth] (Mf) -- (Aps) node[midway, above] {\cite[Theorem 1.2]{0702-4}};
\draw[<->, double, scale=0.8, >=stealth] (Aps) -- (BMO) node[thm] {Theorem \ref{tm-0720-1}};
\draw[<->, double, scale=0.8, >=stealth] (BMO) -- (BMOW) 
node[thm, left, align=center] {Theorem \ref{them-127-2}\\
$0<q<\fz$}
node[thm, right, align=center] {$w\in \mathcal A_{p,\dz}$\\
$1\le p<\fz$};
\draw[<-, scale=0.8, >=stealth] (ABCDE) -- (BMO) node[thm, left] {Factorization};
\end{tikzpicture}
\caption{Equivalence of $\mathcal A_{1,\delta}$ and ${\rm BLO}(\mathbb R^n, \mathcal H_{\infty}^{\delta})$ for all dimension $\delta\in(0,n]$.}
\end{figure}
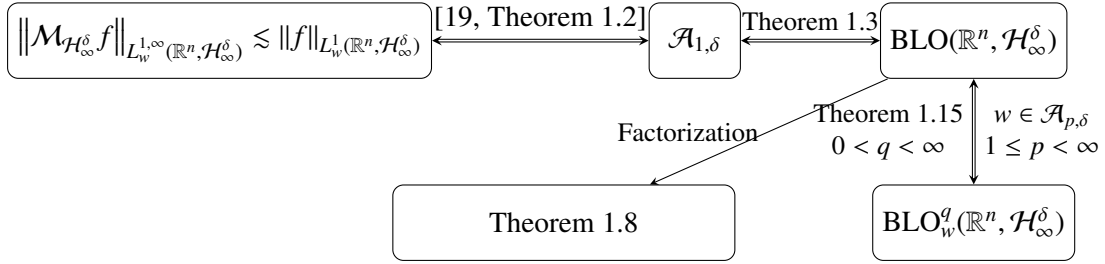

To clarify the motivation driving this paper, we now revisit some history. Recall that, to investigate the regularity of solutions to nonlinear partial differential equations that arise in the study of minimal surfaces, John and Nirenberg \cite{06-30-1} first introduced the BMO (bounded mean oscillation) spaces, in which they established the well-known John--Nirenberg inequality. To be precise, a locally integrable function $f\in\rm{BMO}$ if it satisfies the following condition:
\[\|f\|_{\rm{BMO}}\coloneqq\sup_{Q}\frac{1}{|Q|}\int_Q|f(x)-f_Q|\,dx< \fz,\]
where $f_Q\coloneqq\frac{1}{|Q|}\int_Qf(x)\,dx$ and $Q$ is a cube in $\rn$. Subsequently, many works were inspired by the contributions of John--Nirenberg \cite{06-30-1}. For instance,
Nazarov et al. \cite{0703-3,ntv02} introduced a new BMO space and proved $Tb$-theorem on non-homogeneous spaces; Tolsa \cite{0703-4} investigated BMO and $H^1$ theory under nondoubling measures;
Duong and Yan \cite{dy051,dy052} introduced and developed BMO spaces associated with operators.

In particular, Coifman and Rochberg in their seminal work \cite{06-30-2} established a fundamental relationship between BMO, Muckenhoupt $A_p$ classes and the Hardy–Littlewood maximal operator. In that work, they also introduced the BLO (bounded lower oscillation) space and proved a factorization theorem for BLO involving the $A_1$ class and maximal operators. Recall that a locally integrable function $f\in\rm{BLO}$ if it satisfies
\[\|f\|_{\rm{BLO}}\coloneqq\sup_{Q}\frac{1}{|Q|}\int_Qf(x)-\esinf_{y\in Q}f(y)\,dx< \fz.\]
The space $\rm{BLO}$ is a proper subspace of ${\rm BMO}$ and it is easy to verify that, for any $f\in \rm{BLO}$, $$\|f\|_{\rm{BMO}}\le 2\|f\|_{\rm{BLO}}.$$
Later, Tang \cite{t07} introduced the spaces ${\rm BLO}_{\mathcal{F}}$ based on sections $\mathcal F$ which is connected with the analysis of the Monge--Amp\'ere equation; Nakai and Sadasue \cite{ns17} studied some properties concerning BLO martingales.

$\rm{BMO}$ and $\rm{BLO}$ spaces have profound connections with other branches in harmonic analysis, such as the duality theorem, ${\rm BMO}=(H^1)^*$, with the Hardy space $H^1$ due to Fefferman and Stein \cite{f71,fs72}, and their equivalent characterizations in terms of Muckenhoupt $A_p$ weights. Indeed, John--Nirenberg \cite{06-30-1} (see also the book \cite[Chapter IV]{gr85}) and Coifman--Rochberg \cite{06-30-2} established the following two precise characterizations respectively:
$${\rm{BMO}}=
\{\alpha \ln w: \alpha\ge 0 \ {\text{and}}\  w\in A_p\},~~~\forall p\in(1,\fz)$$
and
$${\rm{BLO}}=\{\alpha \ln w: \alpha\ge 0 \ {\text{and}}\  w\in A_1\}.$$
Here and thereafter, $\ln$ represents the logarithmic function with base $e$, and the \emph{symbol} $A_p$ always denotes the class of  $w\in L_{\loc}^1(\rn)$ (the set of all locally integrable functions on
${{\mathbb R}^n}$) with $w\in(0,\fz)$ almost everywhere such that
\[[w]_{A_p}\coloneqq\sup_Q\lf[\frac{1}{|Q|}\int_Q w (x)\,dx\r]\lf[\frac{1}{|Q|}\int_Q w(x)^{-\frac{1}{p-1}}\,dx\r]^{p-1}<\fz,\quad \forall p\in(1,\fz)\]
with the supremum being taken over all cubes $Q$ and
$$ [w]_{A_1}\coloneqq\esup_{y\in\rn}\frac{\cm w (y)}{w(y)}<\fz,$$
where $\cm$ is the Hardy--Littlewood maximal operator defined by setting
\[\mathcal M(f)(x)\coloneqq\sup_{Q\ni x}\frac1{|Q|}{\int_Q|f(y)|\,dy}
,\quad \forall\,x\in\rn\ \text{and}\ f\in L_{\loc}^1(\rn).\]

Motivated by \cite{06-30-1,06-30-2,0702-4}, we are interested in understanding the relationship between 
$\ca_{p,\dz}$ weights and ${{\rm {BMO}}}(\rn, \ch_{\fz}^{\dz})$ or ${{\rm {BLO}}}(\rn, \ch_{\fz}^{\dz})$ spaces with Hausdorff contents $\ch_{\fz}^{\dz}$ for all  dimension $\delta\in(0,n]$, and consequently to clarify the structure of ${{\rm {BMO}}}(\rn, \ch_{\fz}^{\dz})$ and ${{\rm {BLO}}}(\rn, \ch_{\fz}^{\dz})$. Our main result shows the equivalence between the capacitary Muckenhoupt weights and BMO and BLO spaces with Hausdorff contents, and we further derive factorization theorems for these BMO and BLO spaces.
Recall that the  \emph{space ${{\rm {BMO}}}(\rn, \ch_{\fz}^{\dz})$, with $\delta\in(0,n]$, of functions of bounded $\delta$-dimensional mean oscillation} in $\rn$
is defined by setting
\[{{\rm {BMO}}}(\rn, \ch_{\fz}^{\dz})\coloneqq\lf\{f\in L^1_{\loc}(\rn, \ch_{\infty}^{\delta}):\ \|f\|_{{{\rm {BMO}}}(\rn, \ch_{\fz}^{\dz})}< \fz\r\}\]
with
\[\|f\|_{{{\rm {BMO}}}(\rn, \ch_{\fz}^{\dz})}\coloneqq\sup_{Q}\inf_{c\in \mathbb R}\frac{1}{\ch_{\infty}^{\delta}(Q)}\int_{Q}|f(x)-c|\,d\ch^{\dz}_{\fz},\]
where the supremum is taken over all finite cubes $Q\subset \rn$. Chen and Specter introduced this space in \cite{0623-1} and derived an analogue of the John--Nirenberg inequality for $\rm{BMO}(\rn, \ch_{\fz}^{\dz})$. We also refer to Basak et al. \cite{bcrs25} for the John--Nirenberg inequality concerning translation invariant Hausdorff contents. Additionally, we introduce the \emph{space ${{\rm {BLO}}}(\rn, \ch_{\fz}^{\dz})$ of functions of bounded $\delta$-dimensional lower oscillation} in $\rn$ by setting
\[{{\rm {BLO}}}(\rn, \ch_{\fz}^{\dz})\coloneqq\lf\{f\in L^1_{\loc}(\rn, \ch_{\infty}^{\delta}):\ \|f\|_{{\rm {BLO}}(\rn, \ch_{\fz}^{\dz})}< \fz\r\}\]
with
\[\|f\|_{{{\rm {BLO}}}(\rn, \ch_{\fz}^{\dz})}\coloneqq\sup_{Q}\frac{1}{\ch_{\infty}^{\delta}(Q)}\int_{Q}|f(x)-\esinf_{y\in Q}f(y)|\,d\ch^{\dz}_{\fz},\]
where the supremum is taken over all finite cubes $Q\subset \rn$ and
$$\esinf_{y\in Q} f(y)\coloneqq\inf\lf\{{t\in(0,\fz):\ \mathcal H^{\delta}_\infty}\lf(\{x\in Q: f (x) <t\}\r)>0\r\}.$$

When $\delta=n$ and limiting to Lebesgue measurable functions, the spaces ${{\rm {BMO}}}(\rn, \ch_{\fz}^{\dz})$ and ${{\rm {BLO}}}(\rn, \ch_{\fz}^{\dz})$
go back to the BMO space of John--Nirenberg  \cite{06-30-1} and to the BLO space of Coifman--Rochberg \cite{06-30-2} respectively. This is due to the fact that the Hausdorff content $\ch_\fz^n$
is equivalent to the Lebesgue measure on measurable subsets of $\rn$, namely, there exist positive constants $K_1(n)$ and $K_2(n)$ such that,
for any measurable subset $E\subset \rn$,
$$K_1(n)\ch_\fz^n(E)\le |E|\le K_2(n)\ch_\fz^n(E),$$
which is essential obtained by the equivalence between the $n$-Hausdorff measure and the Lebesgue measure;
see Evans and Gariepy \cite[Chapter 2.2]{eg15} for the detail.

Here the spaces ${{{\rm {BMO}}}(\rn, \ch_{\fz}^{\dz})}$ and ${{{\rm {BLO}}}(\rn, \ch_{\fz}^{\dz})}$ will be regarded as spaces of functions modulo constants as usual. We point out that ${{{\rm {BLO}}}(\rn, \ch_{\fz}^{\dz})}$ is not a vector space and $\|\cdot\|_{{{\rm {BLO}}}(\rn, \ch_{\fz}^{\dz})}$ is not  a norm since it lacks homogeneity for negative constants. But following  Coifman--Rochberg \cite{06-30-2}, we still call them  space and norm. 

In what follows, since the norm appears in the denominator, we always assume that $f\in {\rm BMO}$ or BLO has nonzero norm. In fact, the conclusion is trivial when the norm of $f$ is zero. With these preparations, we now state the first main result as follows.

\begin{thm}\label{tm-0628-1}
Let $\delta\in(0,n]$ and $p\in(1, \fz)$. Then
$f\in {\rm BMO}(\mathbb R^n, \mathcal H_{\infty}^{\delta})$ if and only if $e^{\alpha f}\in \mathcal A_{p,\delta}$  for some non-negative constant $\alpha$, that is,
\begin{enumerate}
\item[{\rm(i)}]
if $w\in \ca_{p,\dz}$, then $\ln w\in {\rm {BMO}}(\rn, \ch_{\fz}^{\dz})$ and there exists a positive constant $C=C(n,\dz, p)$ such that
\begin{equation}\label{eq-0628-1}
\|\ln w\|_{{\rm {BMO}}(\rn, \ch_{\fz}^{\dz})}\le C[w]^{\min\{1,\frac{1}{p-1}\}}_{\ca_{p,\dz}};
\end{equation}
\item[{\rm(ii)}]
conversely, if $f\in {\rm {BMO}}(\rn, \ch_{\fz}^{\dz})$, then there exist positive constants $c_1=c_1(n,\dz, p)$ and $c_2=c_2(n,\dz, p)$ such that
\begin{equation}\label{eq-0628-3}
{\rm{exp}}\lf({\frac{c_1f}{\|f\|_{{\rm {BMO}}(\rn, \ch_{\fz}^{\dz})}}}\r)\in \ca_{p,\dz}\quad and \quad \lf[{\rm{exp}}\lf({\frac{c_1f}{\|f\|_{{\rm {BMO}}(\rn, \ch_{\fz}^{\dz})}}}\r)\r]_{\ca_{p,\dz}}\leq c_2,
\end{equation}
that is, the weight constant has a uniform upper bound with respect to $f\in {\rm {BMO}}(\rn, \ch_{\fz}^{\dz})$ and ${\|f\|_{{\rm {BMO}}(\rn, \ch_{\fz}^{\dz})}}$.
\end{enumerate}
\end{thm}

\begin{rem}
We point out that Theorem \ref{tm-0628-1} returns to the classical Muckenhoupt $A_p$ weight characterization of BMO spaces when $\delta=n$, namely, $${\rm{BMO}}=
\{\alpha \ln w: \alpha\ge 0 \ {\text{and}}\  w\in A_p\},~~~\forall p\in(1,\fz);$$ see Garc\'ia-Cuerva and Rubio de Francia \cite[Chapter IV, p.\,409]{gr85} or Coifman, Rochberg and Weiss \cite{crw76}. But when dimension $\delta\in(0,n)$, the approach in the proof is new. Indeed, the classical approaches are no longer applicable in this case, primarily due to the nonlinearity of Choquet integrals, involved in the capacitary weight and function spaces, in addition to the definition of ${{\rm {BMO}}(\rn, \ch_{\fz}^{\dz})}$.
\end{rem}

Theorem \ref{tm-0628-1} says that capacitary Muckenhoupt weight class $\ca_{p,\dz}$ characterizes ${\rm BMO}(\rn,\ch_\fz^\delta)$ for $p\in(1,\fz)$.
However, $\ca_{1,\dz}$ is unable to characterize ${\rm BMO}(\rn,\ch_\fz^\delta)$.
By establishing an analogue of John--Nirenberg inequality for ${\rm{BLO}}(\rn,\ch_{\fz}^{\dz})$, we then build an equivalent characterization of ${\rm BLO}(\rn,\ch_\fz^\delta)$ via $\ca_{1,\dz}$ as follows.

\begin{thm}\label{tm-0720-1}
Let $\delta\in(0,n]$. Then $f\in {\rm BLO}(\mathbb R^n, \mathcal H_{\infty}^{\delta})$ if and only if $e^{\beta f}\in \mathcal A_{1,\delta}$ for some non-negative $\beta$, that is,
\begin{enumerate}
\item[{\rm(i)}]
if $w\in \ca_{1,\dz}$, then $\ln w\in {{\rm {BLO}}}(\rn, \ch_{\fz}^{\dz})$ and  there exists positive constant $C=C(n, \dz)$ such that
\begin{equation}\label{eq-0628-4}
\|\ln w\|_{{{\rm {BLO}}}(\rn, \ch_{\fz}^{\dz})}\le \ln C+ \ln [w]_{\ca_{1,\dz}};
\end{equation}
\item[{\rm(ii)}] conversely,
if $f\in {{\rm {BLO}}}(\rn, \ch_{\fz}^{\dz})$, then there exist exist positive constants $c_1=c_1(n,\dz)$ and $c_2=c_2(n,\dz)$ such that
\begin{equation}\label{eq-0628-5}
{\rm{exp}}\lf({\frac{c_1f}{\|f\|_{{{\rm {BLO}}}(\rn, \ch_{\fz}^{\dz})}}}\r)\in \ca_{1,\dz}\quad and\quad \lf[{\rm{exp}}\lf({\frac{c_1f}{\|f\|_{{{\rm {BLO}}}(\rn, \ch_{\fz}^{\dz})}}}\r)\r]_{\ca_{1,\dz}}\leq c_2,
\end{equation}
that is, the weight constant has a uniform upper bound with respect to $f\in {{\rm {BLO}}}(\rn, \ch_{\fz}^{\dz})$ and ${\|f\|_{{\rm {BLO}}(\rn, \ch_{\fz}^{\dz})}}$.
\end{enumerate}
\end{thm}

\begin{rem}
We point out that Theorem \ref{tm-0720-1} revisits the classical Muckenhoupt $A_1$-weight characterization of BLO spaces when $\delta=n$, namely, $${\rm{BLO}}=\{\alpha \ln w: \alpha\ge 0 \ {\text{and}}\  w\in A_1\};$$ see Coifman--Rochberg \cite{06-30-2}. But when dimension $\delta\in(0,n)$, the classical methods are no longer applicable and we use new method and technique.
\end{rem}

\begin{rem}
We point out that ${\rm BLO}(\mathbb R^n, \mathcal H_{\infty}^{\delta})$ is properly embedded to ${\rm BMO}(\mathbb R^n, \mathcal H_{\infty}^{\delta})$. Indeed, we show the following embeddings in Proposition \ref{lm-0722-1},
$$L^{\fz}(\rn, \ch_{\fz}^{\dz})\subsetneqq {\rm{BLO}}(\rn, \ch_{\fz}^{\dz})\subsetneqq {\rm{BMO}}(\rn, \ch_{\fz}^{\dz}),\quad 0<\delta\le n.$$
Here and thereafter, $L^{\fz}(E, \ch_\fz^\delta)$ always denotes the space of all functions $f$ on $E$ such that
$$\|f\|_{L^\infty(E,\ch_\fz^\delta)}\coloneqq\inf\lf\{t\in(0,\fz): \ch_\fz^\delta(\{x\in E:\ |f(x)|>t\})=0\r\}<\fz.$$ 
\end{rem}
 
\begin{rem}
From \cite[Proposition 2.17]{0702-4}, we infer that the capacitary
Muckenhoupt weights class $\ca_{p,\dz}$ enjoys the strict monotonicity on the dimension $\dz$, i.e.,
$${\ca_{p,\beta}\subsetneqq \ca_{p,\dz},\quad 0<\beta<\delta\le n},~~1\le p<\fz.$$
This, combined with Theorems \ref{tm-0628-1} and \ref{tm-0720-1},  immediately implies that
the sets ${\rm{BMO}}(\rn, \ch_{\fz}^{\dz})$ and ${\rm{BLO}}(\rn, \ch_{\fz}^{\dz})$ both enjoy the strict monotonicity on the dimension $\dz\in(0,n]$, that is, for any $0<\beta<\delta\le n$,
$${\rm{BMO}}(\rn, \ch_{\fz}^{\beta})\subsetneqq {\rm{BMO}}(\rn, \ch_{\fz}^{\dz})~~\text{and}~~{\rm{BLO}}(\rn, \ch_{\fz}^{\beta})\subsetneqq {\rm{BLO}}(\rn, \ch_{\fz}^{\dz}).$$
Furthermore, by the corresponding John--Nirenberg inequalities, we know that these two embeddings are continuous with respect to the associated norms. The first embedding returns to \cite[Corollary 1.6]{0623-1}.
\end{rem}

As applications of Theorems \ref{tm-0628-1} and \ref{tm-0720-1}, we show the following
factorization theorems of spaces ${{\rm {BMO}}}(\rn, \ch_{\fz}^{\dz})$ and ${\rm BLO}(\rn,\ch_\fz^\delta)$ via capacitary Hardy--Littlewood maximal operators.

\begin{thm}\label{cor-0721-1}
Let $\dz\in(0, n]$.
\begin{enumerate}
\item[{\rm(i)}] If $f\in {{\rm {BMO}}}(\rn, \ch_{\fz}^{\dz})$,
then there exist positive constants $\alpha,\beta$, $b\in L^{\fz}(\rn, \ch_{\fz}^{\dz})$ and
$g_1, g_2\in L^1_{\loc}(\rn,\ch_{\fz}^{\dz})$ such that
    $$f=\alpha \ln \cm_{\ch_{\fz}^{\dz}}g_1-\beta \ln \cm_{\ch_{\fz}^{\dz}}g_2+ b$$
and
$$\alpha+\beta+\|b||_{L^{\fz}(\rn, \ch_{\fz}^{\dz})}\le C\|f\|_{{{\rm {BMO}}}(\rn, \ch_{\fz}^{\dz})},$$
where $C$ is a positive constant depending only on $n$ and $\dz$;
\item[{\rm(ii)}]
conversely, let $\alpha,\beta\in (0,\fz)$, $b\in L^{\fz}(\rn, \ch_{\fz}^{\dz})$ and
$g_1,g_2\in L^1_{\loc}(\rn,\ch_{\fz}^{\dz})$. If for $\ch_\fz^\delta$-almost every $x\in\rn$, $[\cm_{\ch_{\fz}^{\dz}}g_i(x)]<\fz$ with $i\in\{1,2\}$,
then
$$f\coloneqq\alpha \ln \cm_{\ch_{\fz}^{\dz}}g_1-\beta \ln \cm_{\ch_{\fz}^{\dz}}g_2+ b\in {{\rm {BMO}}}(\rn, \ch_{\fz}^{\dz})$$
and there exists a positive constant $C$ depending only on $n$ and $\dz$ such that
$$\|f\|_{{{\rm {BMO}}}(\rn, \ch_{\fz}^{\dz})}\le C\lf(\alpha+\beta+\|b||_{L^{\fz}(\rn, \ch_{\fz}^{\dz})}\r).$$
\end{enumerate}
\end{thm}

\begin{thm}\label{cor-0721-2}
Let $\dz\in(0, n]$.
\begin{enumerate}
\item[{\rm(i)}] If $f\in {{\rm {BLO}}}(\rn, \ch_{\fz}^{\dz})$,
then there exist positive constant $\alpha$, $b\in L^{\fz}(\rn, \ch_{\fz}^{\dz})$ and $g\in L^1_{\loc}(\rn,\ch_{\fz}^{\dz})$ such that $$f=\alpha \ln \cm_{\ch_{\fz}^{\dz}}g+ b$$
and
$$\alpha+\|b||_{L^{\fz}(\rn, \ch_{\fz}^{\dz})}\le C\|f\|_{{{\rm {BLO}}}(\rn, \ch_{\fz}^{\dz})},$$
where $C$ is a positive constant depending only on $n$ and $\dz$;
\item[{\rm(ii)}]
conversely, let $\alpha\in (0,\fz)$, $b\in L^{\fz}(\rn, \ch_{\fz}^{\dz})$ and $g\in L^1_{\loc}(\rn,\ch_{\fz}^{\dz})$.
If for $\ch_\fz^\delta$-almost every $x\in\rn$, $[\cm_{\ch_{\fz}^{\dz}}g(x)]<\fz$, then
$$f\coloneqq\alpha \ln \cm_{\ch_{\fz}^{\dz}}g+ b\in {{\rm {BLO}}}(\rn, \ch_{\fz}^{\dz})$$
and there exists a positive constant $C$ depending only on $n$ and $\dz$ such that
$$\|f\|_{{{\rm {BLO}}}(\rn, \ch_{\fz}^{\dz})}\le C\lf(\alpha+\|b||_{L^{\fz}(\rn, \ch_{\fz}^{\dz})}\r).$$
\end{enumerate}
\end{thm}

\begin{rem}
Theorem \ref{cor-0721-1} states that, for any $\dz\in(0, n]$,
\begin{align*}f\in {{\rm {BMO}}}(\rn, \ch_{\fz}^{\dz}) \iff &\exists\, \alpha,\beta>0, b\in L^{\fz}(\rn, \ch_{\fz}^{\dz})\ \text{and}\
g_1, g_2\in L^1_{\loc}(\rn,\ch_{\fz}^{\dz})\\
&\text{such that}\ f=\alpha \ln \cm_{\ch_{\fz}^{\dz}}g_1-\beta \ln \cm_{\ch_{\fz}^{\dz}}g_2+ b
\end{align*}
and $\|f\|_{{{\rm {BMO}}}(\rn, \ch_{\fz}^{\dz})}$ is equivalent to $\inf\{ \alpha+\beta+\|b||_{L^{\fz}(\rn, \ch_{\fz}^{\dz})}\}$ with the infimum being taken over all $\alpha,\beta$ and $b$ as above. Theorem \ref{cor-0721-2} says that, for any $\dz\in(0, n]$,
\begin{align*}f\in {{\rm {BLO}}}(\rn, \ch_{\fz}^{\dz}) \iff &\exists\, \alpha>0, b\in L^{\fz}(\rn, \ch_{\fz}^{\dz})\ \text{and}\
g\in L^1_{\loc}(\rn,\ch_{\fz}^{\dz})\\
&\text{such that}\ f=\alpha \ln \cm_{\ch_{\fz}^{\dz}}g+ b
\end{align*}
and $\|f\|_{{{\rm {BLO}}}(\rn, \ch_{\fz}^{\dz})}$ is equivalent to  $\inf\{ \alpha+\|b||_{L^{\fz}(\rn, \ch_{\fz}^{\dz})}\}$ with the infimum being taken over all $\alpha$ and $b$ as above. From these two factorizations, it is easy to see
that if $f\in L^1_{\loc}(\rn,\ch_{\fz}^{\dz})$ with $\cm_{\ch_{\fz}^{\dz}}f(x)<\fz$ for $\ch_\fz^\delta$-almost every $x\in\rn$, then
$\ln \cm_{\ch_{\fz}^{\dz}}f$ belongs to both ${{\rm {BMO}}}(\rn, \ch_{\fz}^{\dz})$ and ${{\rm {BLO}}}(\rn, \ch_{\fz}^{\dz})$; furthermore, its norm has a uniform upper bound with respect to $f\in L^1_{\loc}(\rn,\ch_{\fz}^{\dz})$.
\end{rem}

\begin{rem}
We point out that Theorems \ref{cor-0721-1} and \ref{cor-0721-2} coincide with the representation formulae of Coifman--Rochberg \cite{06-30-2} when $\delta=n$, and essentially extend them to the dimension $\delta\in(0,n)$ via some new observations and techniques.
\end{rem}

Although the classical theory of $\rm{BMO}$ spaces has been remarkably successful, with the in-depth development of harmonic analysis and partial differential equations, its limitations have gradually emerged. Especially when dealing with the following problems, the classical framework appears inadequate: weighted inequalities, non-doubling measure spaces and adaptability of differential operators and so on. These requirements have driven the development of weighted $\rm{BMO}$ spaces; see, for example, \cite{b85,0703-2,n97,tz16}, which forms an important research direction in modern harmonic analysis.

In particular, motivated by studying the behavior of the Hilbert transform,  Muckenhoupt and Wheeden \cite{mw76} introduced the weighted BMO space ${\rm BMO}_w$, and proved that $f\in {\rm BMO}$ if and only if $f\in {\rm BMO}_w$ for all $w\in A_{\fz}$. 
This builds another deep connection between Muckenhoupt weights and BMO spaces. Therefore, another objective of this paper is to explore the weighted $\rm{BMO}$ space ${{\rm {BMO}}}^q_{w}(\rn, \ch_{\fz}^{\dz})$ with Hausdorff contents for all dimension $\delta\in(0,n]$. In fact,
we discover that $${\rm{BMO}}(\rn, \ch_{\fz}^{\dz})={{\rm {BMO}}}^q_{w}(\rn, \ch_{\fz}^{\dz})$$
for any $\delta\in(0,n]$, $q\in(0,\fz)$ and $w\in\ca_{p,\dz}$ with $p\in[1,\fz)$, in the sense of equivalent quasi-norms by establishing a capacitary weighted John--Nirenberg inequality. In addition, we also establish the corresponding results for weighted $\rm{BLO}$ space ${{\rm {BLO}}}^q_{w}(\rn, \ch_{\fz}^{\dz})$ for all dimension $\delta\in(0,n]$. These recover several results existing in the literature when $\delta=n$, and generalizes them beyond measure-theoretic frameworks when $\delta\in(0,n)$.

Given a weight $w$ and a subset $E\subset\rn$, let
$w_{\ch_{\fz}^{\dz}}(E)\coloneqq\int_{E}w(x)\,d\ch_{\fz}^{\dz}$
and $p\in(0,\fz)$. The \emph{weighted Choquet--Lebesgue space}
$L^p_w(E,\ch_\fz^\delta)$ is defined as the set of
all functions $f$ satisfying
\begin{align*}
\|f\|_{L^p_w(E,\ch_\fz^\delta)}
\coloneqq\lf\{\int_E |f(x)|^pw(x)\,d\ch_\fz^\delta \r\}^{\frac 1p}<\fz.
\end{align*}
Let $L^p_{w,\loc}(\rn, \ch_{\fz}^{\dz})$ be the set of all functions $f$ satisfying that, for any  compact $K\subset \rn$, $f\in L^p_w(K,\ch_\fz^\delta)$. We introduce the weighted space ${{\rm {BMO}}}^p_{w}(\rn, \ch_{\fz}^{\dz})$ by setting
\[ {{\rm {BMO}}}^p_{w}(\rn, \ch_{\fz}^{\dz})\coloneqq\left\{f\in L^p_{w,\loc}(\rn, \ch_{\infty}^{\delta}): \|f\|_{{{\rm {BMO}}}^p_{w}(\rn, \ch_{\fz}^{\dz})}< \fz\right\},\]
where
\[\|f\|_{{{\rm {BMO}}}^p_{w}(\rn, \ch_{\fz}^{\dz})}\coloneqq\sup_{Q}\inf_{c\in \mathbb R}\lf[\frac{1}{w_{\ch_{\infty}^{\delta}}(Q)}\int_{Q}|f(x)-c|^pw(x)\,d\ch^{\dz}_{\fz}\r]^{\frac 1p}\]
with the supremum being taken over all finite cubes $Q\subset \rn$. If the definition of ${{\rm {BMO}}}^p_{w}(\rn, \ch_{\fz}^{\dz})$ is modified by replacing $\ch_{\fz}^{\dz}$ with dyadic Hausdorff content $\widetilde{\ch}^{\delta}_{\fz}$ (see Section \ref{2.1} for precise definition), then we denote this space as ${{\rm {BMO}}}^p_{w}(\rn, \widetilde{\ch}^{\delta}_{\fz})$. The weighted space ${{\rm {BLO}}}^p_{w}(\rn, \ch_{\fz}^{\dz})$ is defined by setting
\[ {{\rm {BLO}}}^p_{w}(\rn, \ch_{\fz}^{\dz})\coloneqq\left\{f\in L^p_{w,\loc}(\rn, \ch_{\infty}^{\delta}): \|f\|_{{{\rm {BLO}}}^p_{w}(\rn, \ch_{\fz}^{\dz})}< \fz\right\},\]
where
\[\|f\|_{{{\rm {BLO}}}^p_{w}(\rn, \ch_{\fz}^{\dz})}\coloneqq\sup_{Q}\lf[\frac{1}{w_{\ch_{\infty}^{\delta}}(Q)}\int_{Q}|f(x)-\esinf_{y\in Q}f(y)|^pw(x)\,d\ch^{\dz}_{\fz}\r]^{\frac 1p}\]
with the supremum being taken over all finite cubes $Q\subset \rn$. If the definition of ${{\rm {BLO}}}^p_{w}(\rn, \ch_{\fz}^{\dz})$ is adjusted by replacing $\ch_{\fz}^{\dz}$ with $\widetilde{\ch}^{\delta}_{\fz}$, then we denote this space as ${{\rm {BLO}}}^p_{w}(\rn, \widetilde{\ch}^{\delta}_{\fz})$. 

For space ${{\rm {BMO}}}^q_w(\rn, \ch_{\fz}^{\dz})$, we prove the following weighted John--Nirenberg inequality.

\begin{prop}\label{tm-0523-2}
Let $\delta\in(0,n]$, $p\in[1,\fz)$, $q\in(0,\fz)$ and $w\in\ca_{p,\dz}$.
Then there exist positive constants $c$ and $C$ such that, for any cube $Q$, $t\in(0,\fz)$
and $f\in {{\rm {BMO}}}^q_w(\rn, \ch_{\fz}^{\dz})$,
\begin{equation*}
w_{\ch_{\fz}^{\dz}}\lf(\lf\{x\in Q: |f(x)-C_{f,w,Q}|>t\r\}\r)
\le Cw_{\ch_{\fz}^{\dz}}(Q)e^{-\frac{ct}{\|f\|_{{{\rm {BMO}}}^q_w(\rn,\ch_{\fz}^{\dz})}}},
\end{equation*}
where $C_{f,w,Q}\in \Gamma_{f,w,Q}$ defined as in \eqref{gamma}.
\end{prop}

Applying this weighted John--Nirenberg inequality, we show the space ${{\rm {BMO}}}^q_{w}(\rn, \ch_{\fz}^{\dz})$ coinciding with  ${\rm{BMO}}(\rn, \ch_{\fz}^{\dz})$ in the sense of equivalent quasi-norms when $w\in\ca_{p,\dz}$.

\begin{thm}\label{them-0629-1}
Let $\delta\in(0,n]$, $p\in[1,\fz)$, $q\in(0, \fz)$ and $w\in\ca_{p,\dz}$. Then there exists a positive constant $C$ such that, for any $f\in{{{\rm {BMO}}}(\rn, \ch_{\fz}^{\dz})}$,
\begin{equation*}
\frac{1}{C}\|f\|_{{{\rm {BMO}}}(\rn, \ch_{\fz}^{\dz})}\le\|f\|_{{{\rm {BMO}}}_{w}^q(\rn, \ch_{\fz}^{\dz})}\le C\|f\|_{{{\rm {BMO}}}(\rn, \ch_{\fz}^{\dz})}.
\end{equation*}
\end{thm}

\begin{rem}
We point out that, in the case $w\equiv 1$, Theorem \ref{them-0629-1} revisits the following  equivalent quasi-norms characterization of space ${{{\rm {BMO}}}(\rn, \ch_{\fz}^{\dz})}$,
\[\frac{1}{C}\|f\|_{{{\rm {BMO}}}(\rn, \ch_{\fz}^{\dz})}\le\|f\|_{{{\rm {BMO}}}^q(\rn, \ch_{\fz}^{\dz})}\le C\|f\|_{{{\rm {BMO}}}(\rn, \ch_{\fz}^{\dz})};\]
see \cite[Corollary 1.5]{0623-1}. When $\delta=n$,
Theorem \ref{them-0629-1} goes back to the weighted characterization of classical BMO space,
\[\frac{1}{C}\|f\|_{{{\rm {BMO}}}}\le\|f\|_{{{\rm {BMO}}}^q_w}\le C\|f\|_{{{\rm {BMO}}}};\] see Muckenhoupt--Wheeden \cite[Theorem 1.5]{mw76}. In the case $w\equiv 1$ and $\delta=n$, then Theorem \ref{them-0629-1} coincides with the well-known result ${{{\rm {BMO}}}}={{{\rm {BMO}}}^q}$ with $q\in (0,\fz)$ due to  John--Nirenberg \cite{06-30-1}. Thus, the main contribution of this theorem is extending these results to both capacitary weight class $\ca_{p,\dz}$ and nonlinearity integrals setting.
We should also point out that, when the dimension $\delta\in(0,n)$, the classical approaches are no longer applicable.
\end{rem}

For space ${{\rm {BLO}}}^q_w(\rn, \ch_{\fz}^{\dz})$, we prove the following weighted John--Nirenberg inequality.

\begin{prop}\label{tm-127-1}
Let $\delta\in(0,n]$, $p\in[1,\fz)$, $q\in(0,\fz)$ and $w\in\ca_{p,\dz}$.
Then there exist positive constants $c$ and $C$ such that, for any cube $Q$, $t\in(0,\fz)$
and $f\in {{\rm {BLO}}}^q_w(\rn, \ch_{\fz}^{\dz})$,
\begin{equation*}
w_{\ch_{\fz}^{\dz}}\lf(\lf\{x\in Q: f(x)-\esinf_{y\in Q}f(y)>t\r\}\r)
\le Cw_{\ch_{\fz}^{\dz}}(Q)e^{-\frac{ct}{\|f\|_{{{\rm {BLO}}}^q_w(\rn,\ch_{\fz}^{\dz})}}}.
\end{equation*}
\end{prop}

Using this weighted John--Nirenberg inequality, we also show that, for $w\in\ca_{p,\dz}$, the space ${{\rm {BLO}}}^q_{w}(\rn, \ch_{\fz}^{\dz})$ coincides with  ${\rm{BLO}}(\rn, \ch_{\fz}^{\dz})$ in the sense of equivalent quasi-norms as follows.

\begin{thm}\label{them-127-2}
Let $\delta\in(0,n]$, $p\in[1,\fz)$, $q\in(0, \fz)$ and $w\in\ca_{p,\dz}$. Then there exists a positive constant $C$ such that, for any $f\in{{{\rm {BLO}}}(\rn, \ch_{\fz}^{\dz})}$,
\begin{equation*}
\frac{1}{C}\|f\|_{{{\rm {BLO}}}(\rn, \ch_{\fz}^{\dz})}\le\|f\|_{{{\rm {BLO}}}_{w}^q(\rn, \ch_{\fz}^{\dz})}\le C\|f\|_{{{\rm {BLO}}}(\rn, \ch_{\fz}^{\dz})}.
\end{equation*}
\end{thm}

We finish the introduction by briefly sketching out the novelty of this work.
Notice that all results obtained in this paper are valid for all dimension $\delta\in (0,n]$.
As previously stated, these essentially generalize the classical
results which only deal with the case $\delta=n$ within the framework
of linear integrals and measures.
In particular, when dimension $\delta\in (0,n)$, the situation becomes significantly different to settings of Lebesgue integrals and classical Muckenhoupt weight class. In this case, we use different method that avoids the use of absolutely continuity of ``measure", Fubini's theorem, the countable additivity of Lebesgue measures and linearity of the Lebesgue integral as the classical theory.  In other words, we find that these requirements seem superfluous in the corresponding classical theory.
This constitutes the main innovations of this work.

Secondly, recall that the original BMO
space of John--Nirenberg \cite{06-30-1} can also be defined by integral average $f_Q\coloneqq\frac1{|Q|}\int_Q f(x)\,dx$.
This definition is well-suited for proving the $A_p$-weight characterization of BMO, since $A_p$ constants are also defined involving integral averages.
However, the Choquet integral with respect to Hausdorff contents is only essentially defined for non-negative functions.
Therefore, a fundamental question is that there is no suitable integral averages for any general function $f$ in this context, and hence as previously stated, the space ${{{\rm {BMO}}}(\rn, \ch_{\fz}^{\dz})}$ is only defined by minus constant $c$. But,
to show the $A_p$-weight characterization for BMO or $A_1$-weight characterization for BLO, the following Jensen's inequality $e^{f_Q}\le (e^f)_Q$ concerning integral average $f_Q$, i.e.,
\[e^{\frac1{|Q|}\int_Q f(x)\,dx}\le \frac1{|Q|}\int_Q e^{f(x)}\,dx\]
plays a central role. Thus, to show main results the difficulty lies in: How can we define a substitute for the integral average $f_Q$ on Choquet integrals with respect to Hausdorff contents such that: (i) the above Jensen's inequality holds under the substitute of $f_Q$; (ii) the space ${{{\rm {BMO}}}(\rn, \ch_{\fz}^{\dz})}$  introduced by minus constant $c$ and minus the substitute of integral averages coincides.
To overcome this obstacle, we introduce the integral averages $f_{Q,\dz}$ as \eqref{26-1-5} in dyadic Hausdorff contents $\widetilde{\ch}_\fz^{\delta}$ involving the position part and negative part of the function $f$, and establish two Jensen-type inequalities; see Lemma \ref{lm-0708-1}. Furthermore, we also show that the space ${\rm{BMO}}(\rn, \ch_{\fz}^{\dz})$ based on mean oscillations minus the new integral averages is equivalent to the space ${\rm{BMO}}(\rn, \ch_{\fz}^{\dz})$ defined as above; see Proposition \ref{pop-0708-1}, which, together with the Jensen-type inequalities in Lemma \ref{lm-0708-1} and John--Nirenberg inequalities, then connects the  ${\rm{BMO}}(\rn, \ch_{\fz}^{\dz})$ space with capacitary weight $\ca_{p,\dz}$.
This serves as another contribution of this paper.

The paper is organized as follows. In Section \ref{2.1}, we introduce the integral average for any given function $f$ under the Hausdorff content
and then show Jensen-type inequalities related to this integral average. We end this section by proving the equivalent of ${{\rm {BMO}}}(\rn, \ch_{\fz}^{\dz})$ defined by minus constant and by minus the integral averages.
Section \ref{2.2} begins with asserting that the infimum in definition of ${\rm{BMO}}_w^q(\rn,\ch_{\fz}^{\dz})$ is attainable. Then we show the proofs of Propositions \ref{tm-0523-2} and \ref{tm-127-1}. 
Section \ref{s3} focuses on proofs of Theorems \ref{tm-0628-1} and \ref{tm-0720-1}. In Section \ref{s4}, we prove Theorems \ref{cor-0721-1} and \ref{cor-0721-2}. Finally, the proofs of Theorems \ref{them-0629-1} and \ref{them-127-2} are given in Section \ref{s5}.

Throughout the paper, the notation $f\ls g$ (resp. $f\gs g$) means $f\le Kg$ (resp. $f\ge Kg$)
for a positive constant $K$ independent of the main parameters, and
$f\sim g$ amounts to $f\ls g\ls f$. Also, we denote by $\mathbf{1}_E$ the characteristic function of set $E\subset \rn$. Given a cube $Q\subset \rn$ and $t\in(0,\fz)$, the new cube $tQ$ always denotes the cube that shares the same center as $Q$
and whose side length is scaled by a factor of $t$.

\section{Integral average and Jensen-type inequality}\label{2.1}

This section is devoted to setting up the proofs of the main theorems. We first introduce the integral average with respect to the dyadic Hausdorff content
and then show the Jensen-type inequalities related to this integral average.

To this end, we first recall some basic properties of the \emph{dyadic Hausdorff content} $\widetilde{\ch}_\fz^{\delta}$,
which is defined by setting for any subset $E$ of $\rn$,
\begin{align*}
\widetilde{\mathcal{H}}_{\fz}^{\delta }\lf(E \r)
:=\inf\lf\{\sum_i \lf[\emph{l}(Q_{i})\r]^{\delta }:\
E\subset\lf(\bigcup_i Q_{i}\r)^{\circ} \r\},
\end{align*}
where the infimum is taken over all dyadic cubes $\{Q_{i}\}_{i}$ and, for any $G\subset \rn$,
$G^{\circ}$ denote the \emph{interior} of the set $G$.
This dyadic Hausdorff content $\widetilde{\ch}^{\delta}_{\fz}$ is equivalent to $\ch_\fz^{\delta}$, namely,
there exists a positive constant $K(n, \delta)$ such that, for any subset $E\subset \rn$,
\begin{equation}\label{eq0723a}
\ch^{\delta}_{\fz}(E)\le \widetilde{\ch}_\fz^{\delta}(E)\le K(n, \delta)\ch^{\delta}_{\fz}(E).
\end{equation}
We point out that $\widetilde{\ch}^{\delta}_{\fz}$ was proved to be a capacity in the sense of Choquet; see \cite{Ad88} for $\delta\in(n-1,n]$ and \cite{YY08} for $\delta\in(0,n-1]$.

The following are basic properties of Choquet integrals with respect to Hausdorff contents, which will be frequently used in this paper.

\begin{rem}\label{12-4}
Let $\delta\in(0,n]$.
\begin{enumerate}
\item[(i)] By the sublinearity
of the Choquet integral with respect to $\widetilde{\ch}_\fz^\delta$ from \cite[p.\,13, Theorem 1]{0923-2},
we conclude that, for any non-negative functions $f$ and $g$ defined on
$E$,
$$\int_{E}f(x)+g(x)\,d\widetilde{\ch}_\fz^\delta
\le \int_{E}f(x)\,d\widetilde{\ch}_\fz^\delta+\int_{E}f(x)\,d\widetilde{\ch}_\fz^\delta.$$
Moreover,
\begin{align}\label{eq125-0}
\int_E [f(x)+g(x)]\,d\mathcal{H}_{\infty }^{\delta }
\le 2\left[\int_E f(x)\,d\mathcal{H}_{\infty }^{\delta }
+\int_E g(x)\,d\mathcal{H}_{\infty }^{\delta }\right].
\end{align}
\item[(ii)]
(H\"{o}lder's inequality) Let $p\in(1,\infty)$.
Then, for any functions $f$ and $g$ on $\rn$, we have
    \[\int_{\rn}|f(x)g(x)|\,d\widetilde{\ch}_\fz^\delta\leq \lf(\int_{\rn}|f(x)|^p\,d\widetilde{\ch}_\fz^\delta\r)^{\frac{1}{p}}\lf(\int_{\rn}|g(x)|^q\,d\widetilde{\ch}_\fz^\delta\r)^{\frac{1}{q}},\]
where $\frac1p+\frac1q=1$. This follows from the Young inequality and (i); we omit the details.
\item[(iii)] (Minkowski's inequality) Let $p\in(1,\infty)$. Then, for any functions $f$ and $g$ on $\rn$, we have
\[\lf(\int_{\rn}\lf|f(x)+g(x)\r|^p\,d\widetilde{\ch}_\fz^\delta\r)^{\frac1p}\leq \lf(\int_{\rn}\lf|f(x)\r|^p\,d\widetilde{\ch}_\fz^\delta\r)^{\frac1p}
+\lf(\int_{\rn}\lf|g(x)\r|^p\,d\widetilde{\ch}_\fz^\delta\r)^{\frac1p}.\]
This is obtained through (ii).
\item[(iv)] Let $C$ be a positive constant and $E\subset \rn$ with $\widetilde{\ch}_\fz^\delta(E)<\fz$. Then, for any function $f$ on $E$, we have
\[\int_{E}\lf[|f(x)|+C\r]\,d\widetilde{\ch}_\fz^\delta=\int_{E}|f(x)|\,d\widetilde{\ch}_\fz^\delta+C\widetilde{\ch}_\fz^\delta(E).\]
The proof can be found in \cite[Lemma 2.5]{0704-1}. Moreover, it is straightforward to verify that for $\ch_{\fz}^{\dz}$, the equality
\[\int_{E}[|f(x)|+C]\,d\ch_\fz^\delta=\int_{E}|f(x)|\,d\ch_\fz^\delta+C\ch_\fz^\delta(E).\]
\end{enumerate}
\end{rem}

Let $\dz\in(0,n]$. For a given function $f \in L_{\loc}^1(\rn,\widetilde{\ch}_\fz^{\delta})$
and a cube $Q$ of $\rn$, we define the \emph{integral average of $f$ on $Q$ with respect to the dyadic Hausdorff content} by
\begin{align}\label{26-1-5}
f_{Q,\dz}\coloneqq\frac{\int_{Q\cap E_{f,+}} f(x)\,d\widetilde{\ch}_\fz^{\delta}
-\int_{Q\cap E_{f,-}}[-f(x)]\,d\widetilde{\ch}_\fz^{\delta}}{\widetilde{\ch}_\fz^{\delta}(Q\cap E_{f,+})+\widetilde{\ch}_\fz^{\delta}(Q\cap E_{f,-})},
\end{align}
here and thereafter, $E_{f,+}\coloneqq\{x\in\rn: f(x)\ge 0\}$ and $E_{f,-}\coloneqq\{x\in\rn: f(x)<0 \}$.

This integral average plays a very important role in the present paper.
We first build the following two Jensen-type inequalities in the context of the dyadic Hausdorff content $\widetilde{\ch}_\fz^{\delta}$.

\begin{lem}\label{lm-0708-1}
Let $\dz\in(0,n]$ and $f \in L_{\loc}^1(\rn,\widetilde{\ch}_\fz^{\delta})$.
Then, for any cube $Q$ of $\rn$,
\begin{equation}\label{eq-0708-1}
e^{f_{Q,\dz}}\le\frac{1}{\widetilde{\ch}_\fz^{\delta}(Q)}
\lf\{\int_{Q\cap {E_{f,+}}}e^{f(x)}\,d\widetilde{\ch}_\fz^{\delta}+\int_{Q\cap {E_{f,-}}}e^{ f(x)}\,d\widetilde{\ch}_\fz^{\delta}\r\}
\end{equation}
and
\begin{equation}\label{eq-0708-2}
e^{-f_{Q,\dz}}\le\frac{1}{\widetilde{\ch}_\fz^{\delta}(Q)}
\lf\{\int_{Q\cap {E_{f,+}}}e^{- f(x)}\,d\widetilde{\ch}_\fz^{\delta}+\int_{Q\cap {E_{f,-}}}e^{- f(x)}\,d\widetilde{\ch}_\fz^{\delta}\r\}.
\end{equation}
\end{lem}

\begin{rem}\label{rem0704-1}
\begin{enumerate}
\item[(i)] We point out that inequalities \eqref{eq-0708-1} and \eqref{eq-0708-2} may be not equivalent in general.
Indeed, within the context of Hausdorff content $\widetilde{\ch}_\fz^{\delta}$, $(-f)_{Q,\dz}=-f_{Q,\dz}$ may not hold true for some cube $Q$ of $\rn$. For example,
let $n\ge 2$ and $\dz\in (0, n-1]$. Define
    $$Q\coloneqq\underbrace{(0,4] \times (0,4]\times \cdots \times (0,4]}_{n },$$
    $$E\coloneqq\underbrace{(0,1] \times (0,1]\times \cdots \times (0,1]}_{n },$$
    $$F\coloneqq\underbrace{(0,4] \times (0,4]\times \cdots \times (0,4]}_{n-1 }\times (2,4],$$
    and $G\coloneqq Q\setminus(E\cup F).$ Let
$$ f(x)\coloneqq
\begin{cases}
1,\quad &\text{when}\ x\in E,\\
 0,\quad &\text{when}\ x\in G,\\
-2, \quad &\text{when}\ x\in \rn\setminus(E\cup G).
\end{cases}
$$
Then $(-f)_{Q,\dz}\neq-f_{Q,\dz}$. Indeed, by \cite[Remark 2.3(i)]{0702-4}, we know that $\widetilde{\ch}_\fz^{\delta}(E)=1$, $\widetilde{\ch}_\fz^{\delta}(F)=4^{\dz}$, $\widetilde{\ch}_\fz^{\delta}(E\cup G)=4^{\dz}$, $\widetilde{\ch}_\fz^{\delta}(F\cup G)=4^{\dz}$. Thus,
\[f_{Q,\dz}=\frac{\widetilde{\ch}_\fz^{\delta}(E)-2\widetilde{\ch}_\fz^{\delta}(F)}{\widetilde{\ch}_\fz^{\delta}(E\cup G)+\widetilde{\ch}_\fz^{\delta}(F)}=\frac{1-2^{1+2\delta}}{2^{1+2\delta}}\]
and
\[(-f)_{Q,\dz}=\frac{2\widetilde{\ch}_\fz^{\delta}(F)-\widetilde{\ch}_\fz^{\delta}(E)}
{\widetilde{\ch}_\fz^{\delta}(E)+\widetilde{\ch}_\fz^{\delta}(F\cup G)}
=\frac{2^{1+2\delta}-1}{1+ 4^{\dz}}.\]
Therefore, $(-f)_{Q,\dz}\neq-f_{Q,\dz}$.
\item[(ii)] If $f$ is non-negative for $\widetilde{\ch}_\fz^{\delta}$-almost everywhere on $\rn$, then
 $$f_{Q,\delta}=\frac1{\widetilde{\ch}_\fz^{\delta}(Q)}\int_Q|f(x)|\,d\widetilde{\ch}_\fz^{\delta}.$$
In this case, \eqref{eq-0708-1} reduces to the following inequality, for any cube $Q$,
    \begin{equation*}
    e^{\frac{1}{\widetilde{\ch}_\fz^{\delta}(Q)}\int_Q|f(x)|\,d\widetilde{\ch}_\fz^{\delta}}\leq \frac{1}{\widetilde{\ch}_\fz^{\delta}(Q)}\int_Qe^{|f(x)|}\,d\widetilde{\ch}_\fz^{\delta}.
    \end{equation*}
This revisits the classical Jensen's inequality while $\delta=n$.
\end{enumerate}
\end{rem}

We now show Lemma \ref{lm-0708-1}.

\begin{proof}[Proof of  Lemma \ref{lm-0708-1}]
We may assume that $|f(x)|< \fz$ for every $x\in \rn$.
Let $Q$ be a given cube of $\rn$. We next prove \eqref{eq-0708-1} by two cases.

\emph{Case 1)} $f_{Q,\dz}\ge 0$. Using the trivial inequality
\begin{equation}\label{eq-0708-3}
 e^a(s-a+1)\le e^s, \quad  \forall\, a,\ s\in \mathbb R,
\end{equation}
we find that
\[e^{f_{Q,\dz}}f(x)+e^{f_{Q,\dz}}\le e^{f(x)}+f_{Q,\dz}e^{f_{Q,\dz}},\ {\rm for\ any} \ x\in {E_{f,+}} ,\]
and
\[e^{f_{Q,\dz}}\le e^{f(x)}-e^{f_{Q,\dz}}f(x)+f_{Q,\dz}e^{f_{Q,\dz}},\ {\rm for\ any} \ x\in {E_{f,-}} .\]
Subsequently, by (i) and (iv) of Remark \ref{12-4}, we have
\begin{equation}\label{eq-0708-4}
e^{f_{Q,\dz}}\int_{Q\cap {E_{f,+}}}f(x)\,d\widetilde{\ch}_\fz^{\delta}
+e^{f_{Q,\dz}}\widetilde{\ch}_\fz^{\delta}(Q\cap {E_{f,+}})
\le \int_{Q\cap {E_{f,+}}}e^{f(x)}\,d\widetilde{\ch}_\fz^{\delta}+f_{Q,\dz}e^{f_{Q,\dz}}\widetilde{\ch}_\fz^{\delta}(Q\cap {E_{f,+}})
\end{equation}
and
\begin{equation}\label{eq-0708-9}
e^{f_{Q,\dz}}\widetilde{\ch}_\fz^{\delta}(Q\cap {E_{f,-}})\le\int_{Q\cap {E_{f,-}}}e^{f(x)}\,d\widetilde{\ch}_\fz^{\delta}+e^{f_{Q,\dz}}\int_{Q\cap {E_{f,-}}}[-f(x)]\,d\widetilde{\ch}_\fz^{\delta}+f_{Q,\dz}e^{f_{Q,\dz}}\widetilde{\ch}_\fz^{\delta}(Q\cap {E_{f,-}}).
\end{equation}
By the subadditivity of Hausdorff contents, \eqref{eq-0708-4}, \eqref{eq-0708-9} and the definition of $f_{Q,\dz}$, we obtain
\begin{align*}
e^{f_{Q,\dz}}
&\le e^{f_{Q,\dz}}\frac{ {\widetilde{\ch}_\fz^{\delta}(Q\cap {E_{f,+}})+\widetilde{\ch}_\fz^{\delta}(Q\cap {E_{f,-}})}}{\widetilde{\ch}_\fz^{\delta}(Q)}\\
&\leq \frac{1}{\widetilde{\ch}_\fz^{\delta}(Q)}\lf(\int_{Q\cap {E_{f,+}}}e^{f(x)}\,d\widetilde{\ch}_\fz^{\delta}+\int_{Q\cap {E_{f,-}}}e^{f(x)}\,d\widetilde{\ch}_\fz^{\delta}\r)\\
&\quad\quad +\ \frac{e^{f_{Q,\dz}}}{\widetilde{\ch}_\fz^{\delta}(Q)}\Bigg\{f_{Q,\dz}\lf[\widetilde{\ch}_\fz^{\delta}(Q\cap {E_{f,+}})+\widetilde{\ch}_\fz^{\delta}(Q\cap {E_{f,-}})\r]\\
&\quad\quad \lf.-\lf[\int_{Q\cap {E_{f,+}}} f(x)\,d\widetilde{\ch}_\fz^{\delta}-\int_{Q\cap {E_{f,-}}}[-f(x)]\,d\widetilde{\ch}_\fz^{\delta}\r]\r\}\\
&\le\frac{1}{\widetilde{\ch}_\fz^{\delta}(Q)}\lf(\int_{Q\cap {E_{f,+}}}e^{f(x)}\,d\widetilde{\ch}_\fz^{\delta}+\int_{Q\cap {E_{f,-}}}e^{f(x)}\,d\widetilde{\ch}_\fz^{\delta}\r).
\end{align*}

\emph{Case 2)} $f_{Q,\dz}<0$. By \eqref{eq-0708-3} again, we have
\[e^{f_{Q,\dz}}f(x)+e^{f_{Q,\dz}}-f_{Q,\dz}e^{f_{Q,\dz}}\le e^{f(x)},\ {\rm for\ any} \ x\in {E_{f,+}},\]
and
\[e^{f_{Q,\dz}}-f_{Q,\dz}e^{f_{Q,\dz}}\le e^{f(x)}-e^{f_{Q,\dz}}f(x),\ {\rm for\ any} \ x\in {E_{f,-}} .\]
Then, by (i) and (iv) of Remark \ref{12-4} again, we know that
\begin{equation}\label{eq-0708-5}
e^{f_{Q,\dz}}\int_{Q\cap {E_{f,+}}}f(x)\,d\widetilde{\ch}_\fz^{\delta}+e^{f_{Q,\dz}}\widetilde{\ch}_\fz^{\delta}(Q\cap {E_{f,+}})-f_{Q,\dz}e^{f_{Q,\dz}}\widetilde{\ch}_\fz^{\delta}(Q\cap {E_{f,+}})\le \int_{Q\cap {E_{f,+}}}e^{f(x)}\,d\widetilde{\ch}_\fz^{\delta}
\end{equation}
and
\begin{equation}\label{eq-0708-6}
e^{f_{Q,\dz}}\widetilde{\ch}_\fz^{\delta}(Q\cap {E_{f,-}})-f_{Q,\dz}e^{f_{Q,\dz}}\widetilde{\ch}_\fz^{\delta}(Q\cap {E_{f,-}})\le\int_{Q\cap {E_{f,-}}}e^{f(x)}\,d\widetilde{\ch}_\fz^{\delta}+e^{f_{Q,\dz}}\int_{Q\cap {E_{f,-}}}-f(x)\,d\widetilde{\ch}_\fz^{\delta}
\end{equation}
Combining the subadditivity of Hausdorff contents, \eqref{eq-0708-5}, \eqref{eq-0708-6} and the definition of $f_{Q,\dz}$ again, we further conclude that
\begin{align*}
e^{f_{Q,\dz}}
&\le e^{f_{Q,\dz}}\frac{\widetilde{\ch}_\fz^{\delta}(Q\cap {E_{f,+}})+\widetilde{\ch}_\fz^{\delta}(Q\cap {E_{f,-}})}{\widetilde{\ch}_\fz^{\delta}(Q)}\\
&\leq \frac{1}{\widetilde{\ch}_\fz^{\delta}(Q)}\lf(\int_{Q\cap {E_{f,+}}}e^{f(x)}\,d\widetilde{\ch}_\fz^{\delta}+\int_{Q\cap {E_{f,-}}}e^{f(x)}\,d\widetilde{\ch}_\fz^{\delta}\r)\\
&\quad\quad + \frac{e^{f_{Q,\dz}}}{\widetilde{\ch}_\fz^{\delta}(Q)}\Bigg\{f_{Q,\dz}\lf[\widetilde{\ch}_\fz^{\delta}(Q\cap {E_{f,+}})+\widetilde{\ch}_\fz^{\delta}(Q\cap {E_{f,-}})\r]\\
&\quad\quad \lf.-\lf[\int_{Q\cap {E_{f,+}}} f(x)\,d\widetilde{\ch}_\fz^{\delta}-\int_{Q\cap {E_{f,-}}}-f(x)\,d\widetilde{\ch}_\fz^{\delta}\r]\r\}\\
&\le\frac{1}{\widetilde{\ch}_\fz^{\delta}(Q)}\lf(\int_{Q\cap {E_{f,+}}}e^{f(x)}\,d\widetilde{\ch}_\fz^{\delta}+\int_{Q\cap {E_{f,-}}}e^{f(x)}\,d\widetilde{\ch}_\fz^{\delta}\r).
\end{align*}
Therefore, \eqref{eq-0708-1} holds true, and similarly one can prove \eqref{eq-0708-2}.
This finishes the proof of Lemma \ref{lm-0708-1}.
\end{proof}

The following useful corollary is a consequence of Lemma \ref{lm-0708-1}.

\begin{cor}\label{cor-0708-1}
Let $\delta\in(0,n]$ and $p\in(1, \fz)$. Then there exists a positive constant $C$ such that, for any non-negative $f\in \ca_{p,\dz}$,
\begin{equation}\label{eq-0708-a}
\sup_{Q}\frac{1}{\widetilde{\ch}_\fz^{\delta}(Q)}\int_{Q}e^{\ln f(x)-(\ln f)_{Q,\dz}}\,d\widetilde{\ch}_\fz^{\delta}\leq C[f]_{\ca_{p,\dz}}
\end{equation}
and
\begin{equation}\label{eq-0708-q}
\sup_{Q}\frac{1}{\widetilde{\ch}_\fz^{\delta}(Q)}\int_{Q}e^{-\frac{1}{p-1}[\ln f(x)-(\ln f)_{Q,\dz}]}\,d\widetilde{\ch}_\fz^{\delta}\leq C[f]^{\frac{1}{p-1}}_{\ca_{p,\dz}},
\end{equation}
where the supremums are taken over all cubes of $\rn$.
\end{cor}

\begin{proof}
Observe that $p\in(1, \fz)$, $f$ is non-negative and, for any cube $Q$ of $\rn$,
$$\lf(\ln (f^{\frac1{p-1}})\r)_{Q,\delta}=\frac1{p-1}(\ln f)_{Q,\delta}$$
due to $\int_E cf(x)\,d\widetilde{\ch}_\fz^{\delta}= c\int_E f(x)\,d\widetilde{\ch}_\fz^{\delta}$ for any $c\in(0,\fz)$. Applying this and \eqref{eq-0708-2}, we find that
\begin{align*}
e^{-\frac{(\ln f)_{Q,\dz}}{p-1}}
\le \frac{2}{\widetilde{\ch}_\fz^{\delta}(Q)}\int_{Q}f(x)^{-\frac{1}{p-1}}\,d\widetilde{\ch}_\fz^{\delta}.
\end{align*}
From this and \eqref{eq0723a}, we infer that
\begin{align*}
\frac{1}{\widetilde{\ch}_\fz^{\delta}(Q)}\int_{Q}e^{\ln f(x)-(\ln f)_{Q,\delta}}\,d\widetilde{\ch}_\fz^{\delta}
&\ls\frac{1}{\widetilde{\ch}_\fz^{\delta}(Q)}\int_{Q}f(x)\,d\widetilde{\ch}_\fz^{\delta}\lf(\frac{1}{\widetilde{\ch}_\fz^{\delta}(Q)}\int_{Q}f(x)^{-\frac{1}{p-1}}\,d\widetilde{\ch}_\fz^{\delta}\r)^{p-1}\\
&\ls\frac{1}{\ch_\fz^{\delta}(Q)}\int_{Q}f(x)\,d{\ch}_\fz^{\delta}\lf(\frac{1}{{\ch}_\fz^{\delta}(Q)}\int_{Q}f(x)^{-\frac{1}{p-1}}\,d{\ch}_\fz^{\delta}\r)^{p-1}\noz\\
&\ls[f]_{\ca_{p,\dz}}\noz,
\end{align*}
which implies that \eqref{eq-0708-a} holds true. Similarly, by \eqref{eq-0708-1}, we have
\begin{align*}
e^{(\ln f)_{Q,\dz}}
\le \frac{2}{\widetilde{\ch}_\fz^{\delta}(Q)}\int_{Q}f(x)\,d\widetilde{\ch}_\fz^{\delta}
\end{align*}
and hence \eqref{eq-0708-q} also holds true.
This finishes the proof of Corollay \ref{cor-0708-1}.
\end{proof}

We end this section by giving an equivalent quasi-norm of the space ${{\rm {BMO}}}(\rn, \ch_{\fz}^{\dz})$ as follows. This conclusion plays a very important role in the proof of Theorem \ref{tm-0628-1} and also independent of interest.

\begin{prop}\label{pop-0708-1}
Let $\dz\in(0,n]$. Then, for any $f\in {{\rm {BMO}}}(\rn, \ch_{\fz}^{\dz})$, it holds that
$$\|f\|_{{{\rm {BMO}}}(\rn, \ch_{\fz}^{\dz})}
\le \|f\|_{{\widetilde{{\rm {BMO}}}}(\rn, \ch_{\fz}^{\dz})}
\le (1+2K(n,\dz)) \|f\|_{{{\rm {BMO}}}(\rn, \ch_{\fz}^{\dz})}.$$
Here and thereafter, $K(n,\dz)$ is as in \eqref{eq0723a} and
\[\|f\|_{{\widetilde{{\rm {BMO}}}}(\rn, \ch_{\fz}^{\dz})}\coloneqq\sup_{Q}\frac{1}{\ch_{\infty}^{\delta}(Q)}\int_{Q}|f(x)-f_{Q,\dz}|\,d\ch^{\dz}_{\fz},\]
where the supremum is taken over all cubes $Q\subset \rn$.
\end{prop}

\begin{rem}
We point out that when $\delta=n$ and limiting to Lebesgue measurable functions, Proposition \ref{pop-0708-1} coincides with the classical result on BMO spaces and essentially generalizes it to $\delta\in(0,n)$.
In addition, Proposition \ref{pop-0708-1} improves the recent result due to Basak et al. \cite{bcr25}. Recall that Basak et al. in \cite[Theorem 2.4]{bcr25} obtained a similar result as Proposition \ref{pop-0708-1} under the condition $f$ being non-negative function. Here, we get rid of this non-negative restriction condition by introducing the integral average $f_{Q,\delta}$ for general function $f$. 
\end{rem}

\begin{proof}[Proof of Proposition \ref{pop-0708-1}]
Obviously, for any $f\in {\rm {BMO}}(\rn, \ch_{\fz}^{\dz})$,
\[\|f\|_{ {{\rm {BMO}}}(\rn, \ch_{\fz}^{\dz})}\le\|f\|_{{\widetilde{{\rm {BMO}}}}(\rn, \ch_{\fz}^{\dz})}.\]
Conversely, let $E_1\coloneqq Q\cap E_{f,+}$ and $E_2\coloneqq Q\cap E_{f,-}$.
By \eqref{eq0723a}, (i) and (iv) of Remark \ref{12-4}, we find that, for any cube $Q$ and any $c\in\mathbb R$,
\begin{align*}
&\frac{1}{\ch_{\infty}^{\delta}(Q)}\int_{Q}|f(x)-f_{Q,\delta}|\,d\ch^{\dz}_{\fz}\\
&\quad\leq \frac{1}{\ch_{\infty}^{\delta}(Q)}\int_{Q}|f(x)-c|\,d\ch^{\dz}_{\fz}+|f_{Q,\delta}-c|\\
&\quad\leq\frac{1}{\ch_{\infty}^{\delta}(Q)}\int_{Q}|f(x)-c|\,d\ch^{\dz}_{\fz}+\frac{|\int_{E_1} f(x)\,d\widetilde{\ch}_\fz^{\delta}-c\widetilde{\ch}_\fz^{\delta}(E_1)|}{\widetilde{\ch}_\fz^{\delta}(E_1)
+\widetilde{\ch}_\fz^{\delta}(E_2)}+
\frac{|\int_{E_2}[-f(x)]\,d\widetilde{\ch}_\fz^{\delta}
+c\widetilde{\ch}_\fz^{\delta}(E_2)|}{\widetilde{\ch}_\fz^{\delta}(E_1)
+\widetilde{\ch}_\fz^{\delta}(E_2)}\\
&\quad\leq \frac{1}{\ch_{\infty}^{\delta}(Q)}\int_{Q}|f(x)-c|\,d\ch^{\dz}_{\fz}
+\frac{\int_{E_1}|f(x)-c|\,d\widetilde{\ch}_\fz^{\delta}}{\widetilde{\ch}_\fz^{\delta}(E_1)+\widetilde{\ch}_\fz^{\delta}(E_2)}
+\frac{\int_{E_2}|f(x)-c|\,d\widetilde{\ch}_\fz^{\delta}}{\widetilde{\ch}_\fz^{\delta}(E_1)+\widetilde{\ch}_\fz^{\delta}(E_2)}\\
&\quad\leq \frac{1+2K(n,\dz)}{\ch_{\infty}^{\delta}(Q)}\int_{Q}|f(x)-c|\,d\ch^{\dz}_{\fz},
\end{align*}
which implies that
\[\|f\|_{{\widetilde{{\rm {BMO}}}}(\rn, \ch_{\fz}^{\dz})}
\le (1+2K(n,\dz))\|f\|_{ {{\rm {BMO}}}(\rn, \ch_{\fz}^{\dz})}.\]
This finishes the proof of Proposition \ref{pop-0708-1}.
\end{proof}

\section{Capacitary weighted John--Nirenberg inequalities}\label{2.2}

In this section, we give the proofs of Propositions \ref{tm-0523-2} and \ref{tm-127-1}. As direct consequences,
we obtain capacitary John--Nirenberg inequalities for spaces ${{\rm {BMO}}}(\rn, \ch_{\fz}^{\dz})$ and ${{\rm {BLO}}}(\rn, \ch_{\fz}^{\dz})$, which play important role for proving Theorems \ref{tm-0628-1} and \ref{tm-0720-1}, respectively.

To this end, we first give the following conclusion, which asserts that the infimum in definition of ${\rm{BMO}}_w^q(\rn,\ch_{\fz}^{\dz})$ is attainable.

\begin{lem}\label{lm-0522-11}
Let $\dz\in (0, n]$ and $w$ be a weight on $\rn$. If $f\in {{\rm {BMO}}}_w^q(\rn, \ch_{\fz}^{\dz})$ with $q\in(0,\fz)$,
then for any cube $Q$ of $\rn$, there exists a constant $C_{f,w,Q}$ such that
\begin{equation}\label{eq-0522-16}
\frac{1}{w_{\widetilde{\ch}_\fz^{\delta}}(Q)}\int_Q|f(x)-C_{f,w,Q}|^qw(x)\,d\widetilde{\ch}_\fz^{\delta}=\inf_{c\in\mathbb R}\frac{1}{w_{\widetilde{\ch}_\fz^{\delta}}(Q)}\int_Q|f(x)-c|^qw(x)\,d\widetilde{\ch}_\fz^{\delta}.
\end{equation}
\end{lem}

\begin{proof}
Since $f\in {{\rm {BMO}}}_w^q(\rn, \ch_{\fz}^{\dz})$, it follows from \eqref{eq0723a} that, for any cube $Q$ of $\rn$,
\begin{equation}\label{eq-0522-14}
\frac{1}{w_{\widetilde{\ch}_\fz^{\delta}}(Q)}\int_Q|f(x)|^qw(x)\,d\widetilde{\ch}_\fz^{\delta}< \fz
\end{equation}
and
\begin{equation}\label{eq-0522-13}
\inf_{c\in\mathbb R}\frac{1}{w_{\widetilde{\ch}_\fz^{\delta}}(Q)}\int_Q|f(x)-c|^qw(x)\,d\widetilde{\ch}_\fz^{\delta}< \fz.
\end{equation}
Observe that, for any $n\in\nn$, there exists $C_n$ such that
\begin{equation}\label{eq-0626-2}
\frac{1}{w_{\widetilde{\ch}_\fz^{\delta}}(Q)}\int_Q|f(x)-C_n|^qw(x)\,d\widetilde{\ch}_\fz^{\delta}\le\inf_{c\in\mathbb R}\frac{1}{w_{\widetilde{\ch}_\fz^{\delta}}(Q)}\int_Q|f(x)-c|^qw(x)\,d\widetilde{\ch}_\fz^{\delta}+\frac{1}{n}.
\end{equation}
Then, by \eqref{eq-0522-14}, \eqref{eq-0522-13}, \eqref{eq-0626-2} and Remark \ref{12-4}(i), we find that
\begin{align*}
|C_n|^q
&\ls\frac{1}{w_{\widetilde{\ch}_\fz^{\delta}}(Q)}\int_Q|f(x)-C_n|^qw(x)\,d\widetilde{\ch}_\fz^{\delta}
+\frac{1}{w_{\widetilde{\ch}_\fz^{\delta}}(Q)}\int_Q|f(x)|^qw(x)\,d\widetilde{\ch}_\fz^{\delta}\\
&\ls \inf_{c\in\mathbb R}\frac{1}{w_{\widetilde{\ch}_\fz^{\delta}}(Q)}\int_Q|f(x)-c|^qw(x)\,d\widetilde{\ch}_\fz^{\delta}+1
+\frac{1}{w_{\widetilde{\ch}_\fz^{\delta}}(Q)}\int_Q|f(x)|^qw(x)\,d\widetilde{\ch}_\fz^{\delta}
<\fz.
\end{align*}
Thus, $\{C_n\}_{n\in\nn}$ is a bounded sequence in $\mathbb R$, and hence  there exists a convergent subsequence $\{C_{n_k}\}_{k\in\nn}$ with a limit point $C_{f,w,Q}\in\mathbb R$ such that
$$\lim_{k\to\fz}C_{n_k}=C_{f,w,Q}.$$
Moreover, when $q\in(0,1)$, by Remark \ref{eq0723a}(i) and \eqref{eq-0626-2}, we have, for any $k\in\nn$,
\begin{align}\label{eq-0522-15}
&\frac{1}{w_{\widetilde{\ch}_\fz^{\delta}}(Q)}\int_Q|f(x)-C_{f,w,Q}|^qw(x)\,d\widetilde{\ch}_\fz^{\delta}\\
&\quad\le\frac{1}{w_{\widetilde{\ch}_\fz^{\delta}}(Q)}\int_Q|f(x)-C_{n_k}|^qw(x)\,d\widetilde{\ch}_\fz^{\delta}
+\frac{1}{w_{\widetilde{\ch}_\fz^{\delta}}(Q)}\int_Q|C_{n_k}-C_{f,w,Q}|^qw(x)\,d\widetilde{\ch}_\fz^{\delta}\noz\\
&\quad\le\inf_{c\in\mathbb R}\frac{1}{w_{\widetilde{\ch}_\fz^{\delta}}(Q)}\int_Q|f(x)-c|^qw(x)\,d\widetilde{\ch}_\fz^{\delta}+\frac{1}{n_{k}}
+|C_{n_k}-C_{f,w,Q}|^q\noz.
\end{align}
When $q\in [1,\fz)$, by Remark \ref{12-4}(iii) and \eqref{eq-0626-2}, we get, for any $k\in\nn$,
\begin{align}\label{eq-0626-1}
&\lf[\frac{1}{w_{\widetilde{\ch}_\fz^{\delta}}(Q)}\int_Q|f(x)-C_{f,w,Q}|^qw(x)\,d\widetilde{\ch}_\fz^{\delta}\r]^{\frac{1}{q}}\\
&\quad\le\lf[\frac{1}{w_{\widetilde{\ch}_\fz^{\delta}}(Q)}\int_Q|f(x)-C_{n_k}|^qw(x)\,d\widetilde{\ch}_\fz^{\delta}\r]^{\frac1q}
+\lf[\frac{1}{w_{\widetilde{\ch}_\fz^{\delta}}(Q)}\int_Q|C_{n_k}-C_{f,w,Q}|^qw(x)\,d\widetilde{\ch}_\fz^{\delta}\r]^{\frac{1}{q}}\noz\\
&\quad\le\lf[\inf_{c\in\mathbb R}\frac{1}{w_{\widetilde{\ch}_\fz^{\delta}}(Q)}\int_Q|f(x)-c|^qw(x)\,d\widetilde{\ch}_\fz^{\delta}+\frac{1}{n_{k}}\r]^{\frac{1}{q}}
+|C_{n_k}-C_{f,w,Q}|\noz.
\end{align}
Taking $k\to\fz$ on both \eqref{eq-0522-15} and \eqref{eq-0626-1},
we immediately obtain $\eqref{eq-0522-16}$. This completes the proof of Lemma \ref{lm-0522-11}.
\end{proof}

The value at which the infimum is attained in Lemma \ref{lm-0522-11} exists, but it may be not unique. The subsequent result further investigates the property of the set
consisting of all points at which the infimum is achieved.

\begin{lem}
Let $\dz\in (0, n]$, $q\in[1,\fz)$ and $w$ be a weight on $\rn$.
For any $f\in {{\rm {BMO}}}_w^q(\rn, \ch_{\fz}^{\dz})$, define
\begin{equation}\label{gamma}
\Gamma_{f,w,Q}\coloneqq\lf\{C_{f,w,Q}\in\rr:\ C_{f,w,Q}\ {\rm is\ as\ in\ Lemma\ \ref{lm-0522-11}}\ \r\}.
\end{equation}
Then the set $\Gamma_{f,w,Q}$ is a certain closed interval or single point.
\end{lem}

\begin{proof}
To see this, for any given cube $Q$ of $\rn$, let
\[F(c)\coloneqq\frac{1}{w_{\widetilde{\ch}_\fz^{\delta}}(Q)}\int_Q|f(x)-c|^qw(x)\,d\widetilde{\ch}_\fz^{\delta},
\quad \forall\,c\in \rr.\]
Then, by Remark \ref{12-4}(iii), we have, for any $c_1,\ c_2\in\mathbb R$,
\[F(c_2)^{\frac{1}{q}}\leq \lf(\frac{1}{w_{\widetilde{\ch}_\fz^{\delta}}(Q)}
\int_Q\lf(|f(x)-c_1|+|c_1-c_2|\r)^qw(x)\,d\widetilde{\ch}_\fz^{\delta}\r)^{\frac{1}{q}}\leq F(c_1)^{\frac{1}{q}}+|c_1-c_2|,\]
which implies that $F^{\frac{1}{q}}$ is  uniformly continuous on $\rr$.
Thus, $\Gamma_{f,w,Q}$ is a closed set.

By Lemma \ref{lm-0522-11}, we know that $\Gamma_{f,w,Q}\neq \emptyset$.
Without loss of generality, we may assume that there exist $c_1,c_2\in\Gamma_{f,w,Q}$ with $c_1<c_2$. Then the interval
$[c_1,c_2]\subset \Gamma_{f,w,Q}$. Indeed, for any $c\in(c_1, c_2)$ we can choose $\lambda\in(0, 1)$
such that $c=\lambda c_2+(1-\lambda)c_1$, then by Remark \ref{12-4}(iii), we have
\begin{align*}
F(c)^{\frac{1}{q}}
&\leq \lambda F(c_2)^{\frac{1}{q}}+(1-\lambda)F(c_1)^{\frac{1}{q}}\\
&=\lf\{\inf_{c\in\mathbb R}\frac{1}{w_{\widetilde{\ch}_\fz^{\delta}}(Q)}\int_Q|f(x)-c|^qw(x)\,d\widetilde{\ch}_\fz^{\delta}\r\}^{\frac{1}{q}},
\end{align*}
which implies that $c\in\Gamma_{f,w,Q}$. In conclusion, $\Gamma_{f,w,Q}$ is either a closed interval or a single point.

Next, we show that $\Gamma_{f,w,Q}$ may be a closed interval in a special case by an example.
Let $\dz=n=q=1$, $w\equiv1$, cube $Q\coloneqq [0, 2)$ and
$$f(x)\coloneqq 
\begin{cases}
2,\quad &\text{when}\ x\in[0,1),\\
0,\quad &\text{when}\ x\in \mathbb R\setminus[0,1).
\end{cases}
$$
Obviously, we find that $\widetilde{\ch}_\fz^{\delta}([0,1))=\widetilde{\ch}_\fz^{\delta}([1,2))=1$ and $\widetilde{\ch}_\fz^{\delta}([0,2))=2$.
Thus, it is not difficult to calculate that
$$F(c)\coloneqq 
\begin{cases}
1,\quad &\text{when}\ c\in[0,2],\\
|c-1|,\quad &\text{when}\ c\in \mathbb R\setminus[0,2],
\end{cases}
$$
which implies that $\Gamma_{f,w,Q}=[0, 2]$.
\end{proof}

To prove Proposition \ref{tm-0523-2}, we still need the following properties of capacitary Muckenhoupt weights established in \cite{0702-4}.

\begin{lem}\label{lm0704-1}
Let $\delta\in(0,n]$, $p\in[1,\fz)$ and $ w \in \mathcal A_{p,\delta}$.
\begin{enumerate}
\item[\rm{ (i)}]
For any cube $Q$ of $\rn$ and any subset $E\subset Q$,
\begin{equation*}
\lf[\frac{\mathcal H^\delta_\infty(E)}{\mathcal H^\delta_\infty(Q)}\r]^p
\le 2^p[w]_{\ca_{p,\delta}}\frac{ w _{\mathcal H^\delta_\infty}(E)}{ w _{\mathcal H^\delta_\infty}(Q)};
\end{equation*}
moreover, for any $t\in[1,\fz)$, we have
$ w _{\mathcal H^{\delta}_\infty}(tQ)\le 2^p[w]_{\ca_{p,\delta}} t^{p\delta} w _{\mathcal H^{\delta}_\infty}(Q)$.

\item[\rm{ (ii)}]
If $q\in[p,\fz)$, then $ w \in \mathcal A_{q,\delta}$
and $[w]_{\ca_{q, \dz}}\le 2^{q-1}[w]_{\ca_{p, \dz}}$.

\item[\rm {(iii)}]
If $p\in(1,\fz)$, then $ w^{-\frac{1}{p-1}} \in \mathcal A_{\frac{p}{p-1},\delta}$ and $[w^{-\frac{1}{p-1}}]_{\mathcal A_{\frac{p}{p-1},\delta}}=[w]_{\ca_{p,\dz}}^{\frac{1}{p-1}}$.

\item[\rm{(iv)}]
If $s\in(0,1)$, then $w^s\in\ca_{q,\dz}$ and $[w^s]_{\ca_{q, \dz}}\le 2[w]^s_{\ca_{p, \dz}}$, where $q=sp+1-s$.

\item[\rm{ (v)}]
There exist $\varepsilon\in(0,1)$ and $C>0$ such that, for any cube $Q$ and any subset $E$ of $Q$,
\[\frac{w_{\ch_{\fz}^{\dz}}(E)}{w_{\ch_{\fz}^{\dz}}(Q)}\leq C\lf(\frac{\ch_{\fz}^{\dz}(E)}{\ch_{\fz}^{\dz}(Q)}\r)^{\varepsilon}.\]

\item[\rm{ (vi)}]
There exists a constant $\gamma\in(0,1)$ such that $w^{1+\gamma}\in \ca_{p,\delta}$.
\end{enumerate}
\end{lem}

\begin{proof}
(i) follows from \cite[Lemma 2.4]{0702-4}; (ii) and (iv) can be easily obtained by the definition of capacitary Muckenhoupt weight and H\"older's inequality:
for any functions $f$ and $g$ on $\rn$,
\begin{align}\label{12e1}
\int_{\rn}|f(x)g(x)|\,d\ch_\fz^\delta\leq 2\lf(\int_{\rn}|f(x)|^p\,d\ch_\fz^\delta\r)^{\frac{1}{p}}\lf(\int_{\rn}|g(x)|^{p'}\,d\ch_\fz^\delta\r)^{\frac{1}{p'}},
\end{align}
where $p\in(1,\infty)$ and $\frac1p+\frac1{p'}=1$; see \cite{0923-2}.
(iii) is self-evident and (vi) is just \cite[Corollary 1.8]{0702-4}. We now proceed to prove (v).
According to the reverse H\"older inequality of the capacitary Muckenhoupt weight
 obtained in \cite[Theorem 1.7]{0702-4}, there exist positive constants $K$ and $\gamma$ such that, for every cube $Q$,
\[\lf[\frac{1}{\ch_{\infty}^{\delta}(Q)}\int_{Q}w(x)^{1+\gamma}\,d\ch_{\infty}^{\delta}\r]^{\frac{1}{1+\gamma}}\leq \frac{K}{\ch_{\infty}^{\delta}(Q)}\int_{Q}w(x)\,d\ch_{\infty}^{\delta}.\]
Form this and H\"older's inequality \eqref{12e1} again, we deduce that
\begin{align*}
\frac{w_{\ch_{\fz}^{\dz}}(E)}{w_{\ch_{\fz}^{\dz}}(Q)}&=\frac{1}{w_{\ch_{\fz}^{\dz}}(Q)}\int_Qw(x)\mathbf{1}_E(x)\,d\ch_{\fz}^{\dz}\\
&\leq 2\frac{1}{w_{\ch_{\fz}^{\dz}}(Q)}\lf(\int_{Q}w(x)^{1+\gamma}\,d\ch_{\infty}^{\delta}\r)^{\frac{1}{1+\gamma}}\ch_{\fz}^{\dz}(E)^{\frac{\gamma}{\gamma+1}}\ls \lf(\frac{\ch_{\fz}^{\dz}(E)}{\ch_{\fz}^{\dz}(Q)}\r)^{\varepsilon},
\end{align*}
where $\varepsilon=\frac{\gamma}{\gamma+1}$. This finishes the proof of Lemma \ref{lm0704-1}.
\end{proof}

Let $\dz\in (0, n]$ and $Q$ be a cube in $\rn$. We denote by $\mathcal D(Q)$ the set of all dyadic cubes generated by the cube $Q$. Furthermore, we denote by $\mathcal D_0(Q)$ the collection of all dyadic subcubes of $Q$, i.e., $\mathcal D_0(Q)\coloneqq \{P\in \mathcal D(Q): P\subset Q\}$. It is not difficult to see that when $Q$ is a dyadic cube in $\rn$, $\mathcal D(Q)$ represents the collection of all dyadic cubes in $\rn$. 

The following lemma is a variant of \cite[Lemma 2.11]{0702-4}. The proof is similar, the details being omitted.
\begin{lem}\label{lem-1203-1}
Let $\delta\in(0,n]$, $p\in [1,\fz)$, $w\in \ca_{p,\delta}$, $Q$ be a given cube of $\rn$ and $E\subset Q$. Then there exist a family $\{Q_j\}_{j\in\nn}\subset \mathcal D_0(Q)$ of non-overlapping dyadic subcubes of $Q$ and a subset $F\subset Q$ with $\ch_\fz^\delta(F)=0$ such that
$$E\subset\lf(\bigcup_{j\in\nn} Q_j\r) \bigcup F\quad {\rm and} \quad \sum_{j\in\nn} \int_{Q_j}w(x)\,d\ch_{\fz}^{\delta}\leq K[w]_{\ca_{p,\delta}} \int_Ew(x)\,d\ch_{\fz}^{\delta},$$
where $K$ is a positive constant independent of $E$ and $w$.
\end{lem}

By Lemma \ref{lem-1203-1} and the proof of \cite[Proposition 2.10]{0702-4}, we obtain the following lemma.

\begin{lem}\label{lm124-1}
Let $\delta\in(0,n]$, $p\in [1,\fz)$, $w\in \ca_{p,\delta}$, $Q$ be a given cube of $\rn$ and $E\subset Q$. Then there exist a subset $F\subset Q$ and a family $\{Q_j\}_{j\in\nn}\subset \mathcal D_0(Q)$ of non-overlapping dyadic subcubes of $Q$ such that
\begin{enumerate}
\item[{\rm(i)}]
$E\subset (\bigcup_{j\in\nn} Q_j)\cup F$ and $w_{{\ch}_\fz^{\delta}}(F)=0$;
\item[{\rm(ii)}]
$\sum_{j\in\nn}w_{{\ch}_\fz^{\delta}}(Q_j)\le K[w]_{\ca_{p,\delta}}w_{{\ch}_\fz^{\delta}}(E)$;
\item[{\rm(iii)}]
for any $j\in\nn$, we have
$w_{{\ch}_\fz^{\delta}}(Q_j)\le 2K[w]_{\ca_{p,\delta}}w_{{\ch}_\fz^{\delta}}(Q_j\cap E)$,
\end{enumerate}
where $K$ is as in Lemma \ref{lem-1203-1}.
\end{lem}

From Lemma \ref{lm124-1}, we infer the following pointwise estimate, which is used to establish the Calder\'on--Zygmund decompsition.
\begin{lem}\label{lm124-2}
Let $\dz\in (0, n]$, $p\in [1, \fz)$, $w\in\ca_{p,\dz}$ and $Q$ be a given cube of $\rn$. If $f\in L^{1}(Q,w_{\ch_{\fz}^{\dz}})$, then for $w_{\ch_{\fz}^{\dz}}$-almost everywhere $x\in Q$,
\begin{align}\label{eq1114-2}
\cm_{w_{\ch_{\fz}^{\dz}}}^{{\rm{d}}, Q}f(x)\ge \frac{1}{4K[w]_{\ca_{p,\delta}}}|f(x)|,
\end{align}
where $K$ is as in Lemma \ref{lem-1203-1} and the locally dyadic maximal function $\cm_{w_{\ch_{\fz}^{\dz}}}^{{\rm{d}}, Q}$ is defined by
\[\cm_{w_{\ch_{\fz}^{\dz}}}^{{\rm{d}}, Q}f(x)\coloneqq \sup_{P\ni x, P\in \mathcal D_0(Q)}\frac{1}{w_{\ch_{\fz}^{\dz}}(P)}\int_{P}|f(y)|\,dw_{\ch_{\fz}^{\dz}}, \quad \forall\ x\in Q,\]
with the supremum being taken over all dyadic subcubes $P$ of $Q$ containing $x$.
\end{lem}

\begin{proof}
To prove the lemma, we may assume that exists $x_0\in Q$ such that $0<|f(x_0)|<\fz$. Otherwise,
\[Q=\{x\in Q: |f(x)|=0\}\cup\{x\in Q: |f(x)|=\fz\}.\]
Since $f\in L^1(Q, w_{\ch_{\fz}^{\dz}})$, we have
\begin{align}\label{eq124-4}
w_{\ch_{\fz}^{\dz}}(\{x\in Q: |f(x)|=\fz\})=0.
\end{align}
Thus, for $w_{\ch_{\fz}^{\dz}}$-almost everywhere $x\in Q$,
\[x\in\lf\{x\in Q: |f(x)|=0\r\}\subset \lf\{x\in Q: \cm_{w_{\ch_{\fz}^{\dz}}}^{{\rm{d}}, Q}f(x)\ge \frac{1}{4K[w]_{\ca_{p,\delta}}}|f(x)|\r\},\]
which implies \eqref{eq1114-2}.

For each $k\in\zz$, define
\begin{align}\label{eq1114-5}
E_k\coloneqq \lf\{x\in Q: 2^{k-1}|f(x_0)|\leq |f(x)|<2^{k}|f(x_0)|\r\}.
\end{align}
Then clearly,
\begin{align}\label{eq1114-3}
Q&=\bigcup_{k\in\zz}E_k \cup \lf\{x\in Q: |f(x)|=0\r\}\cup \lf\{x\in Q: |f(x)|=\fz\r\}\\
&=\bigcup_{k\in\zz}E_k\cup \lf\{x\in Q: \cm_{w_{\ch_{\fz}^{\dz}}}^{{\rm{d}}, Q}f(x)\ge \frac{1}{4K[w]_{\ca_{p,\delta}}}|f(x)|\r\} \cup \lf\{x\in Q: |f(x)|=\fz\r\}.\noz
\end{align}
We now consider two cases based on the value of $w_{{\ch}_\fz^{\delta}}(E_k)$.

\emph{\textbf{Case 1)}} If $w_{{\ch}_\fz^{\delta}}(E_k)=0$, set $F_k\coloneqq E_k$.

\emph{\textbf{Case 2)}} If $w_{{\ch}_\fz^{\delta}}(E_k)>0$, applying Lemma \ref{lm124-1} to $E_k$, then there exist a subset $F_k\subset Q$ and a family $\{Q_{k,j}\}_{j\in\nn}\subset\mathcal{D}_0(Q)$ of non-overlapping cubes in $Q$ such that
\begin{enumerate}
\item[{\rm(i)}]
$E_k\subset (\bigcup_{j\in\nn} Q_{k,j})\cup F_k$ and $w_{{\ch}_\fz^{\delta}}(F_k)=0$;
\item[{\rm(ii)}]
$\sum_{j\in\nn}w_{{\ch}_\fz^{\delta}}(Q_{k,j})\le K[w]_{\ca_{p,\delta}}w_{{\ch}_\fz^{\delta}}(E_k)$;
\item[{\rm(iii)}]
for any $j\in\nn$, we have
$w_{{\ch}_\fz^{\delta}}(Q_{k,j})\le 2K[w]_{\ca_{p,\delta}}w_{{\ch}_\fz^{\delta}}(Q_{k,j}\cap E_k)$.
\end{enumerate}
From (i), for any $x\in E_k$, either $x\in\cup_{j\in\nn}Q_{k,j}$ or $x\in F_k$. If $x\in\cup_{j\in\nn}Q_{k,j}$, then there exists $j_0$ such that $x\in Q_{k,j_0}$. By \eqref{eq1114-5}, (iii) and $x\in E_k$, we have
\begin{align*}
\cm_{w_{\ch_{\fz}^{\dz}}}^{{\rm{d}}, Q}f(x)
&\ge \frac{1}{w_{\ch_{\fz}^{\dz}}(Q_{k, j_0})}\int_{Q_{k, j_0}}|f(y)|\,dw_{\ch_{\fz}^{\dz}}\ge \frac{1}{w_{\ch_{\fz}^{\dz}}(Q_{k, j_0})}\int_{Q_{k, j_0}\cap E_k}|f(y)|\,dw_{\ch_{\fz}^{\dz}}\\
&\ge \frac{w_{\ch_{\fz}^{\dz}}(Q_{k, j_0}\cap E_k)}{w_{\ch_{\fz}^{\dz}}(Q_{k, j_0})}2^{k-1}|f(x_0)|\ge \frac{1}{4K[w]_{\ca_{p,\delta}}}|f(x)|,
\end{align*}
which implies
$$x\in \lf\{y\in Q:\ \cm_{w_{\ch_{\fz}^{\dz}}}^{{\rm{d}}, Q}f(y)\ge \frac{1}{4K[w]_{\ca_{p,\delta}}}|f(y)|\r\}.$$
Therefore,
\begin{align}\label{eq1114-7}
E_k\subset \lf\{x\in Q:\ \cm_{w_{\ch_{\fz}^{\dz}}}^{{\rm{d}}, Q}f(x)\ge \frac{1}{4K[w]_{\ca_{p,\delta}}}|f(x)|\r\}\cup F_k.
\end{align}
Combining \eqref{eq1114-3},\eqref{eq1114-7} and \textbf{Case 1}, we have
\[Q=\bigcup_{k\in\zz} F_k\cup \lf\{x\in Q: \cm_{w_{\ch_{\fz}^{\dz}}}^{{\rm{d}}, Q}f(x)\ge \frac{1}{4K[w]_{\ca_{p,\delta}}}|f(x)|\r\}\cup \lf\{x\in Q: |f(x)|=\fz\r\}.\]
Note that, by (i) and  \eqref{eq124-4}, we have
\[w_{\ch_{\fz}^{\dz}}\lf(\bigcup_{k\in\zz} F_k\cup \lf\{x\in Q: |f(x)|=\fz\r\}\r)\leq \sum_{k\in\zz}w_{\ch_{\fz}^{\dz}}(F_k)+w_{\ch_{\fz}^{\dz}}\lf(\lf\{x\in Q: |f(x)|=\fz\r\}\r)=0.\]
Hence, \eqref{eq1114-2} holds for $\ch_{\fz}^{\dz}$-almost everywhere $x\in Q$. This completes the proof of Lemma \ref{lm124-2}.
\end{proof}

The following useful lemma is just \cite[Proposition 2.7]{0702-4}.
\begin{lem}\label{lm-0919-4}
Let $\delta\in(0,n]$, $p\in[1,\fz)$ and $ w \in \mathcal A_{p,\delta}$. Then there exists a positive constant $C$ such that,
for any $f\in L_w^1(\rn,\ch_\fz^\delta)$,
\[\frac{1}{4}\int_{\mathbb R^n}|f(x)|w(x)\,d\mathcal H^{\delta}_\infty
\leq\int_\rn |f(x)|\,dw_{\mathcal H^{\delta}_\infty}
\leq C[w]_{\ca_{p,\delta}}^{\frac{1}{p}}\int_{\mathbb R^n}|f(x)|w(x)\,d\mathcal H^{\delta}_\infty.\]
\end{lem}

Using Lemmas \ref{lm124-2} and \ref{lm-0919-4} and some standard processes, we establish the following capacitary weighted Calder\'on--Zygmund decomposition, which goes back to \cite[Theorem 3.4]{0623-1} by taking $w\equiv 1$.

\begin{lem}\label{lm-410-5} (Capacitary weighted Calder\'on--Zygmund decomposition)
Let $\delta\in(0, n]$, $w\in \ca_{p,\delta}$ with $p\in[1,\fz)$ and
 $f\in L^1_{w}(Q,\widetilde{\ch}_\fz^{\delta})$ with a given cube $Q$ of $\rn$. Then, for any $$\lambda\ge\frac{1}{w_{\widetilde{\ch}_\fz^{\delta}}(Q)}\int_{Q}|f(x)|w(x)\,d\widetilde{\ch}_\fz^{\delta},$$ there exist a collection of pairwise disjoint dyadic subcubes $\{Q_j\}_{j\in\nn}$ of $Q$, and positive constants $C_0,\,C_1\in(1,\fz)$ such that
\begin{enumerate}
\item[{\rm(i)}]
for any $j\in\nn$,
$$\lz<\frac{1}{w_{\widetilde{\ch}_\fz^{\delta}}(Q_j)}\int_{Q_j}|f(x)|w(x)\,d\widetilde{\ch}_\fz^{\delta}\le C_0\lambda;$$
\item[{\rm(ii)}]
for $w_{\widetilde{\ch}_\fz^{\delta}}$-almost every $x\in Q\backslash\cup_{j\in\nn}Q_j$, $|f(x)|\leq C_1\lambda$.
\end{enumerate}
Moreover, for any dyadic subcube $Q^*$ of $Q$, if there exists some $Q_j$ satisfying $Q_j\subset Q^*$, then
\begin{align}\label{12e2}
\frac{1}{w_{\widetilde{\ch}_\fz^{\delta}}(Q^*)}\int_{Q^*}|f(x)|w(x)\,d\widetilde{\ch}_\fz^{\delta}\le C_0\lambda.
\end{align}
\end{lem}

\begin{proof}
Let $\{R_i\}_{i}$ be the collection of all the dyadic subcubes of $Q$ which satisfies
\begin{align*}
\frac{1}{w_{\widetilde{\ch}_\fz^{\delta}}(R_i)}\int_{R_i}|f(x)|w(x)\,d\widetilde{\ch}_\fz^{\delta}>\lambda,
\end{align*}
and $\{Q_j\}_j$ be the maximal subcollection of $\{R_i\}_i$. Then, for each $Q_j$, its parent cube $Q'_j$ satisfies
\begin{align}\label{eq-0523-8}
\frac{1}{w_{\widetilde{\ch}_\fz^{\delta}}(Q'_j)}\int_{Q'_j}|f(x)|w(x)\,d\widetilde{\ch}_\fz^{\delta}\le\lambda,
\end{align}
moreover, $Q'_j\subset 3Q_j$. Thus, by Lemma \ref{lm0704-1}(i) and \eqref{eq0723a},
there exists a positive constant $C_0$ such that
\begin{align*}
w_{\widetilde{\ch}_\fz^{\delta}}(Q'_j)\le C_0 w_{\widetilde{\ch}_\fz^{\delta}}(Q_j).
\end{align*}
From this and \eqref{eq-0523-8}, we deduce that
\[\frac{1}{w_{\widetilde{\ch}_\fz^{\delta}}(Q_j)}\int_{Q_j}|f(x)|w(x)\,d\widetilde{\ch}_\fz^{\delta}\le \frac{C_0}{w_{\widetilde{\ch}_\fz^{\delta}}(Q'_j)}\int_{Q'_j}|f(x)|w(x)\,d\widetilde{\ch}_\fz^{\delta}\le C_0 \lambda.\]
Thus, (i) holds. Then, for any $x\in Q\setminus\cup_{j}Q_j$, if $P\in\mathcal D_0(Q)$ with $x\in P$,
then $P$ is not in the selected cubes $\{R_i\}_i$, and hence
\[\frac{1}{w_{\widetilde{\ch}_\fz^{\delta}}(P)}\int_{P}|f(x)|w(x)\,d\widetilde{\ch}_\fz^{\delta}\le\lambda.\]
By this, Lemma \ref{lm-0919-4} and \eqref{eq0723a}, we have,
\begin{align*}
\frac{1}{w_{{\ch}_\fz^{\delta}}(P)}\int_{P}|f(x)|\,dw_{{\ch}_\fz^{\delta}}&\ls\frac{1}{w_{{\ch}_\fz^{\delta}}(P)}\int_{P}|f(x)|w(x)\,d{\ch}_\fz^{\delta}\\
&\ls \frac{1}{w_{\widetilde{\ch}_\fz^{\delta}}(P)}\int_{P}|f(x)|w(x)\,d\widetilde{\ch}_\fz^{\delta}\ls\lambda,
\end{align*}
which implies that $\cm_{w_{\ch_{\fz}^{\dz}}}^{{\rm{d}}, Q}f(x)\ls\lambda$. Subsequently, by Lemma \ref{lm124-2}, we have, for $w_{\widetilde{\ch}_\fz^{\delta}}$-almost every $x\in Q\backslash\cup_{j\in\nn}Q_j$,
\[|f(x)|\ls \cm_{w_{\ch_{\fz}^{\dz}}}^{{\rm{d}}, Q}f(x)\ls \lambda,\]
which implies (ii). Finally, by the maximality of $\{Q_j\}_j$, we obtain \eqref{12e2}. This completes the proof of Lemma \ref{lm-410-5}.
\end{proof}

The following important lemma comes from \cite[Lemma 2.6 and Proposition 2.12]{0702-4}.

\begin{lem}\label{lm-0919-2}
Let $\delta\in(0,n]$ and $w$ be a non-negative function on $\rn$.
For any given cube $Q$ of $\rn$, if $\{Q_j\}_{j\in\nn}$ is a family of non-overlapping dyadic subcubes of $Q$, then there exist a subfamily $\{Q_{j_v}\}$ of $\{Q_j\}_{j\in\nn}$ and a sequence $\{Q^*_k\}$ of non-overlapping dyadic subcubes of $Q$, such that
$$\bigcup_{j\in\nn}Q_j\subset \lf(\bigcup_{v}Q_{j_v}\r)\bigcup\lf(\bigcup_{k}Q^*_k\r)$$
and, for any cubes $Q^*_k$, there exists some $Q_j$ such that $Q_j\subset Q^*_k$, and
$$w _{\widetilde{\ch}_\fz^{\delta}} \lf(Q^*_k\r)
\le \sum_{Q_{j_v}\subset Q^*_k} w _{\widetilde{\ch}_\fz^{\delta}}(Q_{j_v}).$$
Moreover, if $f\in L^1_w({\bigcup_{v} Q_{j_v}},\widetilde{\ch}_\fz^{\delta})$ with $w\in \mathcal A_{p,\delta}$, then there exists a positive constant $C_2$ such that
\begin{equation*}
\sum_{v}\int_{Q_{j_v}}|f(x)|w(x)\,d\widetilde{\ch}_\fz^{\delta}\le C_2\int_{\bigcup_{v} Q_{j_v}}|f(x)|w(x)\,d\widetilde{\ch}_\fz^{\delta}.
\end{equation*}
\end{lem}

Now we are ready to prove Proposition \ref{tm-0523-2}.

\begin{proof}[Proof of Proposition \ref{tm-0523-2}]
Without loss of generality, by \eqref{eq0723a}, we may assume  $\|f\|_{{{\rm {BMO}}}^q_w(\rn, \widetilde{\ch}_{\fz}^{\dz})}=1$ and for any cube $Q$ of $\rn$,
 $C_{f,w,Q}=0$, where $C_{f,w,Q}\in \Gamma_{f,w,Q}$ is as in \eqref{gamma}.
In this case, we only need to prove that there exist positive constants $c$ and $C$ such that for any cube $Q$ of $\rn$ and $t\in(0,\fz)$,
\begin{align}\label{eq-0524-2}
w_{\widetilde{\ch}_{\fz}^{\dz}}(\{x\in Q:\ |f(x)|>t\})\le Ce^{-ct}w_{\widetilde{\ch}_{\fz}^{\dz}}(Q).
\end{align}

We claim that for any $s>1$, $k\in \nn$ and $t>(k-1)C_{\dz}s$,
\begin{align}\label{eq-0622-a}
w_{\widetilde{\ch}_{\fz}^{\dz}}(\{x\in Q:\ |f(x)|>t\})\le F_k(t) w_{\widetilde{\ch}_{\fz}^{\dz}}(Q),
\end{align}
where $C_{\dz}=[(1+2^{q-1})(1+C_0)C_1]^{\frac{1}{q}}$ with $C_0$ and $C_1$ as in Lemma \ref{lm-410-5}, $F_1(t)=\frac1{t^q}$ and, when $k\ge 2$, $F_k(t)=\frac{2C_2}{s^q}F_{k-1}(t-C_{\dz}s)$ with $C_2$ as in Lemma \ref{lm-0919-2}.

We prove the above claim by induction. Since $C_{f,w,Q}=0$ and $\|f\|_{{{\rm {BMO}}}^q_w(\rn, \widetilde{\ch}_{\fz}^{\dz})}=1$, it follow that
\begin{align}\label{eq-0627-1}
\frac{1}{w_{\widetilde{\ch}_\fz^{\delta}}(Q)}\int_Q|f(x)|^qw(x)\,d\widetilde{\ch}_\fz^{\delta}\le 1.
\end{align}
By this and Chebychev's inequality, for any $t>0$, we have
\[w_{\widetilde{\ch}_{\fz}^{\dz}}(\{x\in Q: |f(x)|>t\})\le \frac{1}{t^q}\int_{Q}|f(x)|^qw(x)\,d\widetilde{\ch}_{\fz}^{\dz}\le \frac{1}{t^q}w_{\widetilde{\ch}_{\fz}^{\dz}}(Q),\]
which implies \eqref{eq-0622-a} holds true  when $k=1$. Now, we assume that, when $k=i$ with $i\in\nn$, \eqref{eq-0622-a} holds true, and we will prove that when $k=i+1$, \eqref{eq-0622-a}  also holds true.

For any $s\ge 1$, by \eqref{eq-0627-1} and Lemma \ref{lm-410-5}, there exist a collection of pairwise disjoint dyadic subcubes $\{Q_j\}_{j\in\nn}$ of $Q$ and positive constants $C_0\in(1,\fz)$ and $C_1\in(1,\fz)$ such that
\begin{enumerate}
\item[{\rm(i)}]
for any $j\in\nn$, $s^q<\frac{1}{w_{\widetilde{\ch}_\fz^{\delta}}(Q_j)}\int_{Q_j}|f(x)|^qw(x)\,d\widetilde{\ch}_\fz^{\delta}\leq C_0s^q$;
\item[{\rm(ii)}]
$|f(x)|\leq C_1^{\frac{1}{q}}s$ \ for $w_{\widetilde{\ch}_\fz^{\delta}}$-almost every $x\in Q\backslash\cup_{j\in\nn}Q_j$;
\item[{\rm(iii)}]
For any dyadic subcube $Q^*$ of $Q$, if there exists some $Q_j$ such that $Q_j\subset Q^*$, then
\[\frac{1}{w_{\widetilde{\ch}_\fz^{\delta}}(Q^*)}\int_{Q^*}|f(x)|^qw(x)\,d\widetilde{\ch}_\fz^{\delta}\le C_0s^q.\]
\end{enumerate}
The fact that
\[t> iC_{\dz}s>C_1^{\frac{1}{q}}s\]
implies
\begin{align}\label{eq-0524-4}
\{x\in Q:\ |f(x)|>t\}\subset \bigcup_{j\in\nn}Q_j\cup N,
\end{align}
where for some subset $N\subset Q$ with $w_{\widetilde{\ch}_\fz^{\delta}}(N)=0$.

For $\{Q_j\}_{j\in\nn}$, applying Lemma \ref{lm-0919-2}, there exist a subfamily $\{Q_{j_v}\}$ of $\{Q_j\}_{j\in\nn}$ and a sequence $\{Q^*_k\}$ of non-overlapping dyadic subcubes of $Q$, such that
\begin{enumerate}
\item[{\rm(iv)}]
$$\bigcup_{j\in\nn}Q_j\subset \lf(\bigcup_{v}Q_{j_v}\r)\cup\lf(\bigcup_{k}Q^*_k\r);$$

\item[{\rm(v)}] for any cubes $Q^*_k$, there exists some $Q_j$ such that $Q_j\subset Q_k^*$, and
$$w _{\widetilde{\ch}_\fz^{\delta}} \lf(Q^*_k\r)
\le \sum_{Q_{j_v}\subset Q^*_k} w _{\widetilde{\ch}_\fz^{\delta}}(Q_{j_v});$$
\item[{\rm(vi)}] for some positive constant $C_2$,
\begin{equation*}
\sum_{v}\int_{Q_{j_v}}|f(x)|^qw(x)\,d\widetilde{\ch}_\fz^{\delta}\le C_2\int_{\cup_{v} Q_{j_v}}|f(x)|^qw(x)\,d\widetilde{\ch}_\fz^{\delta}.
\end{equation*}
\end{enumerate}
Therefore, by \eqref{eq-0524-4}, (ii) and (iv), we obtain
\begin{align}\label{eq-0524-5}
&w_{\widetilde{\ch}_\fz^{\delta}}(\{x\in Q:\ |f(x)|>t\})\\
&\quad\leq \sum_{v}w_{\widetilde{\ch}_\fz^{\delta}}(\{x\in Q_{j_v}:\ |f(x)|>t\})+\sum_{k}w_{\widetilde{\ch}_\fz^{\delta}}(\{x\in Q^*_{k}:\ |f(x)|>t\}).\noz
\end{align}
For any dyadic subcube $Q'$ of $Q$, let $C_{Q'}\coloneqq C_{f,w, Q'}\in \Gamma_{f, w, Q'}$.
Then, by (iii), for any $Q'\in\{Q_{j_v},Q_{k}^*\}$, we have
\[\frac{1}{w_{\widetilde{\ch}_\fz^{\delta}}(Q')}\int_{Q'}|f(x)|^qw(x)\,d\widetilde{\ch}_\fz^{\delta}\le C_0s^q.\]
Combining this with $\|f\|_{{{\rm {BMO}}}^q_w(\rn, \widetilde{\ch}_{\fz}^{\dz})}=1$ and the fact $q\in(0,\fz)$, we conclude that
\begin{align*}
|{C_{Q'}}|^q&=\frac{1}{w_{\widetilde{\ch}_\fz^{\delta}}(Q')}\int_{Q'}|C_{Q'}|^qw(x)\,d\widetilde{\ch}_\fz^{\delta}\\
&\leq(1+ 2^{q-1})\lf\{\frac{1}{w_{\widetilde{\ch}_\fz^{\delta}}(Q')}\lf(\int_{Q'}|f(x)-C_{Q'}|^qw(x)\,d\widetilde{\ch}_\fz^{\delta}+
\int_{Q'}|f(x)|^qw(x)\,d\widetilde{\ch}_\fz^{\delta}\r)\r\}\noz\\
&\le C_{\dz}^qs^q.\noz
\end{align*}
Therefore, for any $x\in Q$,
\begin{align}\label{eq-0524-7}
|f(x)|\leq |f(x)-C_{Q'}|+|C_{Q'}|\leq |f(x)-C_{Q'}|+C_{\dz}s,
\end{align}
where $Q'\in\{Q_{j_v},Q_{k}^*\}$. Note that when $k=i+1$ and $t>(k-1)C_{\dz}s$, then $t-C_{\dz}s>(i-1)C_{\dz}s$. Subsequently, by \eqref{eq-0524-5}, \eqref{eq-0524-7}, \eqref{eq-0622-a} with $k=i$, $\|f-C_{Q'}\|_{{{\rm {BMO}}}^q_w(\rn, \widetilde{\ch}_{\fz}^{\dz})}=\|f\|_{{{\rm {BMO}}}^q_w(\rn, \widetilde{\ch}_{\fz}^{\dz})}=1$, we have
\begin{align*}
&w_{\widetilde{\ch}_\fz^{\delta}}\lf(\lf\{x\in Q: |f(x)|>t\r\}\r)\\
&\quad\leq \sum_{v}w_{\widetilde{\ch}_\fz^{\delta}}\lf(\lf\{x\in Q_{j_v}: |f(x)-C_{Q_{j_v}}|>t-C_{\dz}s\r\}\r)+\sum_{k}w_{\widetilde{\ch}_\fz^{\delta}}\lf(\lf\{x\in Q^*_{k}: |f(x)-C_{Q_{k}^*}|>t-C_{\dz}s\r\}\r)\noz\\
&\quad\leq F_i(t-C_{\dz}s)\lf(\sum_{v}w_{\widetilde{\ch}_\fz^{\delta}}(Q_{j_v})+\sum_{k}w_{\widetilde{\ch}_\fz^{\delta}}(Q^*_{k})\r)\noz.
\end{align*}
This, together with (v), (i), (vi) and \eqref{eq-0627-1}, further implies that
\begin{align*}
&w_{\widetilde{\ch}_\fz^{\delta}}\lf(\lf\{x\in Q: |f(x)|>t\r\}\r)\\
&\quad\leq F_i(t-C_{\dz}s)\lf(\sum_{v}w_{\widetilde{\ch}_\fz^{\delta}}(Q_{j_v})+\sum_{k}\sum_{Q_{j_v}\subset Q^*_k} w _{\widetilde{\ch}_\fz^{\delta}}(Q_{j_v})\r)\noz\\
&\quad\leq 2F_i(t-C_{\dz}s)\sum_{v}w_{\widetilde{\ch}_\fz^{\delta}}(Q_{j_v})\leq F_i(t-C_{\dz}s)\sum_{v}\frac{2}{s^q}\int_{Q_{j_v}}|f(x)|^qw(x)\,d\widetilde{\ch}_\fz^{\delta}\noz\\
&\quad\leq F_i(t-C_{\dz}s)\frac{2C_2}{s^q}\int_{\cup_{v} Q_{j_v}}|f(x)|^qw(x)\,d\widetilde{\ch}_\fz^{\delta}\leq F_i(t-C_{\dz}s)\frac{2C_2}{s^q}\int_{Q}|f(x)|^qw(x)\,d\widetilde{\ch}_\fz^{\delta}\noz\\
&\quad\leq F_i(t-C_{\dz}s)\frac{2C_2}{s^q}w_{\widetilde{\ch}_{\fz}^{\dz}}(Q)= F_{i+1}(t)w_{\widetilde{\ch}_{\fz}^{\dz}}(Q)\noz,
\end{align*}
which implies when $k=i+1$, \eqref{eq-0622-a} also holds true.

Finally, choosing $s=(2C_2)^{\frac2q}$, for any $t\in(0, \fz)$, if $t>C_{\dz}s$, then there exists a $k\in\nn$ such that
\[kC_{\dz}s<t\le (k+1)C_{\dz}s.\]
Thus, applying \eqref{eq-0622-a}, we have
\begin{align*}
&w_{\widetilde{\ch}_\fz^{\delta}}\lf(\lf\{x\in Q:\ |f(x)|>t\r\}\r)\\
&\quad\leq w_{\widetilde{\ch}_\fz^{\delta}}\lf(\lf\{x\in Q:\ |f(x)|>kC_{\dz}s\r\}\r)\leq F_k(kC_{\dz}s)w_{\widetilde{\ch}_{\fz}^{\dz}}(Q)\\
&\quad= \lf(\frac{2C_2}{s^q}\r)^{k-1}F_1(C_{\dz}s)w_{\widetilde{\ch}_{\fz}^{\dz}}(Q)
=\frac{1}{C_{\dz}^q}e^{-\frac{q(k+1)\ln s}{2}}w_{\widetilde{\ch}_{\fz}^{\dz}}(Q)\leq \frac{1}{C_{\dz}^q}e^{-\frac{\ln 2C_2}{C_{\dz}(2C_2)^{\frac{2}{q}}}t}w_{\widetilde{\ch}_{\fz}^{\dz}}(Q)
\noz.
\end{align*}
If $t\leq C_{\dz}s$, then use the trivial estimate
\[w_{\widetilde{\ch}_\fz^{\delta}}\lf(\lf\{x\in Q: |f(x)|>t\r\}\r)\leq e^{C_{\dz}s}e^{-t}w_{\widetilde{\ch}_{\fz}^{\dz}}(Q).\]
Therefore, for any $t>0$, \eqref{eq-0524-2} holds true. This completes the proof of Proposition \ref{tm-0523-2}.
\end{proof}

As an immediate corollary of Proposition \ref{tm-0523-2} by taking $w\equiv 1$ and $q=1$,
we have the following John--Nirenberg inequality for space ${\rm{BMO}}(\rn,\ch_{\fz}^{\dz})$. This revisits \cite[Theorem 1.3]{0623-1} and plays an important role in proof of Theorem \ref{tm-0628-1}.

\begin{cor}\label{lm-0522-17}
Let $\delta\in(0,n]$. Then there exist positive constants $c=c(n, \dz)$ and $C=C(n,\dz)$ such that, for any cube $Q$ of $\rn$, $t\in(0, \fz)$ and $f\in {\rm {BMO}}(\rn, \ch_{\fz}^{\dz})$,
\[\ch_{\fz}^{\dz}(\{x\in Q: |f(x)-C_{Q}|>t\})\le C\ch_{\fz}^{\dz}(Q)e^{-\frac{ct}{\|f\|_{{\rm {BMO}}(\rn, \ch_{\fz}^{\dz})}}}\]
where $C_Q\coloneqq C_{f,w,Q}$ as in \eqref{eq-0522-16} with $w\equiv 1$ and $q=1$.
\end{cor}

Next we are ready to prove Proposition \ref{tm-127-1}.

\begin{proof}[Proof of Proposition \ref{tm-127-1}]
Without loss of generality, by \eqref{eq0723a}, we may assume  $\|f\|_{{{\rm {BLO}}}^q_w(\rn, \widetilde{\ch}_{\fz}^{\dz})}=1$.
In this case, we only need to prove that there exist positive constants $c$ and $C$ such that,
for any cube $Q$ of $\rn$ and $t\in(0,\fz)$,
\begin{align}\label{eq128-3}
w_{\widetilde{\ch}_{\fz}^{\dz}}\lf(\lf\{x\in Q: f(x)-\esinf_{y\in Q}f(y)>t\r\}\r)\le Ce^{-ct}w_{\widetilde{\ch}_{\fz}^{\dz}}(Q).
\end{align}
We point out that the main idea to prove \eqref{eq128-3} is similar to the proof of Proposition \ref{tm-0523-2}, with
$|f(x)-C_{f,w,Q}|$ therein replaced by $f(x)-\esinf_{y\in Q}f(y)$. Thus, we only give the details for the differences.

We also prove that, for any $s>1$, $k\in \nn$ and $t>(k-1)C_{\dz}s$,
\begin{align}\label{eq128-4}
w_{\widetilde{\ch}_{\fz}^{\dz}}\lf(\lf\{x\in Q: f(x)-\esinf_{y\in Q}f(y)>t\r\}\r)\le F_k(t) w_{\widetilde{\ch}_{\fz}^{\dz}}(Q)
\end{align}
by induction,
where $C_{\dz}\coloneqq [(2^q+2)C_0C_1]^{\frac{1}{q}}$ with $C_0$ and $C_1$ as in Lemma \ref{lm-410-5}, $F_1(t)=\frac1{t^q}$ and, when $k\ge 2$, $F_k(t)=\frac{2C_2}{s^q}F_{k-1}(t-C_{\dz}s)$ with $C_2$ as in Lemma \ref{lm-0919-2}.

Indeed, it is not difficulty to find that \eqref{eq128-4} holds true for $k=1$,
and we assume that, when $k=i$ with $i\in\nn$, \eqref{eq128-4} holds true.
By an argument similar to that used in the proof of \eqref{eq-0524-5}, we conclude that there exist disjoint dyadic subcubes $\{Q_j\}_{j\in\nn}$ of $Q$, $\{Q_{j_v}\}_v$ and $\{Q_k^*\}_k$ satisfying (i)-(vi) in the proof of \eqref{eq-0524-5}
but with $f(x)$ replaced by $f(x)-\esinf_{y\in Q}f(y)$. Using those notation, we have
\begin{align}\label{eq128-7}
&w_{\widetilde{\ch}_\fz^{\delta}}\lf(\lf\{x\in Q: f(x)-\esinf_{y\in Q}f(y)>t\r\}\r)\\
&\quad\leq \sum_{v}w_{\widetilde{\ch}_\fz^{\delta}}\lf(\lf\{x\in Q_{j_v}: f(x)-\esinf_{y\in Q}f(y)>t\r\}\r)\noz\\
&\quad\quad+\sum_{k}w_{\widetilde{\ch}_\fz^{\delta}}\lf(\lf\{x\in Q^*_{k}: f(x)-\esinf_{y\in Q}f(y)>t\r\}\r).\noz
\end{align}

Now we prove that for any $Q'\in\{Q_{j},Q_{k}^*\}$,
\begin{align}\label{eq128-2}
\esinf_{y\in Q'}f(y)-\esinf_{y\in Q}f(y)\leq C_{\dz}s
\end{align}
and hence
\begin{align}\label{eq128-a}
f-\esinf_{y\in Q}f(y)\le f-\esinf_{y\in Q'}f(y)+ C_{\dz}s.
\end{align}
Since $Q_j\subset Q$, we have
\[\int_{Q_j}|f(x)-\esinf_{y\in Q_j}f(y)|^qw(x)\,d\widetilde{\ch}_\fz^{\delta}\leq \int_{Q_j}|f(x)-\esinf_{y\in Q}f(y)|^qw(x)\,d\widetilde{\ch}_\fz^{\delta}.\]
Hence, using this together with (i), we have
\begin{align*}
&\lf(\esinf_{y\in Q_j}f(y)-\esinf_{y\in Q}f(y)\r)^q=\frac{1}{w_{\widetilde{\ch}_\fz^{\delta}}(Q_j)}\int_{Q_j}|\esinf_{y\in Q_j}f(y)-\esinf_{y\in Q}f(y)|^qw(x)\,d\widetilde{\ch}_\fz^{\delta}\\
&\quad\leq \frac{2^{q-1}+1}{w_{\widetilde{\ch}_\fz^{\delta}}(Q_j)}\lf(\int_{Q_j}|f(x)-\esinf_{y\in Q_j}f(y)|^qw(x)\,d\widetilde{\ch}_\fz^{\delta}+\int_{Q_j}|f(x)-\esinf_{y\in Q}f(y)|^qw(x)\,d\widetilde{\ch}_\fz^{\delta}\r)\\
&\quad\leq \frac{2^{q}+2}{w_{\widetilde{\ch}_\fz^{\delta}}(Q_j)}\int_{Q_j}|f(x)-\esinf_{y\in Q}f(y)|^qw(x)\,d\widetilde{\ch}_\fz^{\delta}\leq (2^q+2)C_0s^q,
\end{align*}
which implies that \eqref{eq128-2} holds for $Q'\in \{Q_j\}$. Moreover, for any $Q'\in\{Q_k^*\}$, by (iv), there exists some $Q_j$ such that $Q_j\subset Q'$. Subsequently, \[\esinf_{y\in Q'}f(y)-\esinf_{y\in Q}f(y)\leq \esinf_{y\in Q_j}f(y)-\esinf_{y\in Q}f(y)\leq C_{\dz}s,\]
which implies that \eqref{eq128-2} holds for $Q'\in \{Q^{*}_k\}$.

Finally, applying \eqref{eq128-a} and an argument similar to that used in the proof of Proposition \ref{tm-0523-2} after \eqref{eq-0524-7}, we obtain \eqref{eq128-3}.
This completes the proof of Proposition \ref{tm-127-1}.
\end{proof}

As a direct consequence of Proposition \ref{tm-127-1} by taking $w\equiv 1$ and $q=1$,
we have the following John--Nirenberg inequality for space ${\rm{BLO}}(\rn,\ch_{\fz}^{\dz})$.

\begin{cor}\label{them-0627-1}
Let $\delta\in(0,n]$ and $f\in {{\rm {BLO}}}(\rn, \ch_{\fz}^{\dz})$.
Then, for any cube $Q$ and  $t\in(0,\fz)$,
\begin{equation*}
\ch_{\fz}^{\dz}\lf(\lf\{x\in Q:\ f(x)-\esinf_{y\in Q}f(y)>t\r\}\r)\le C\ch_{\fz}^{\dz}(Q)e^{-\frac{ct}{\|f\|_{{{\rm {BLO}}}(\rn, \ch_{\fz}^{\dz})}}},
\end{equation*}
where $c$ and $C$ are positive constants independent of $f$, $Q$ and $t$.
\end{cor}

When this article nearly completing, we learned from Professor Daniel Spector that Basak et al. \cite{bcr25} independently obtained a local version of John--Nirenberg inequality for the space ${\rm{BLO}}(Q_0,\ch_{\fz}^{\dz})$ with $Q_0$ being a finite cube of $\rn$. Their inequality is in fact equivalent to
Corollary \ref{them-0627-1}.

As an immediately consequence of Corollary \ref{them-0627-1}, the following equivalent quasi-norm for ${{\rm {BLO}}}(\rn, \ch_{\fz}^{\dz})$ is obtained, which can also be regarded as a corollary of Theorem \ref{them-127-2} by taking $w\equiv 1$.

\begin{cor}
Let $\delta\in(0,n]$ and $q\in(0,\fz)$.
Then there exists a positive constant $C$ such that, for any $f\in {{\rm {BLO}}}(\rn, \ch_{\fz}^{\dz})$,
\[\frac1 C\|f\|_{{{\rm {BLO}}}(\rn, \ch_{\fz}^{\dz})}\le \|f\|_{{{\rm {BLO}}}^q(\rn, \ch_{\fz}^{\dz})}\le C \|f\|_{{{\rm {BLO}}}(\rn, \ch_{\fz}^{\dz})},\]
where \[\|f\|_{{{\rm {BLO}}}^q(\rn, \ch_{\fz}^{\dz})}\coloneqq \sup_{Q}\left[\frac{1}{\ch_{\infty}^{\delta}(Q)}\int_{Q}|f(x)-\esinf_{y\in Q}f(y)|^q\,d\ch^{\dz}_{\fz}\right]^{\frac1q}.\]
\end{cor}

\section{Proofs of Theorems \ref{tm-0628-1} and \ref{tm-0720-1}}\label{s3}

We first prove Theorem \ref{tm-0628-1}.

\begin{proof}[Proof of Theorem \ref{tm-0628-1}]
(i) We first prove that when $w\in\ca_{p,\dz}$ with $p=2$, \eqref{eq-0628-1} holds true. To this end, by Corollary \ref{cor-0708-1} with $p=2$, for any cube $Q$ of $\rn$, we have
\begin{equation}\label{eq-0522-18}
\frac{1}{\widetilde{\ch}_\fz^{\delta}(Q)}\int_{Q}e^{\ln w(x)-(\ln w)_{Q,\dz}}\,d\widetilde{\ch}_\fz^{\delta}\ls[w]_{\ca_{2,\dz}}
\end{equation}
and
\begin{equation}\label{eq-0522-19}
\frac{1}{\widetilde{\ch}_\fz^{\delta}(Q)}\int_{Q}e^{(\ln w)_{Q,\dz}-\ln w(x)}\,d\widetilde{\ch}_\fz^{\delta}\ls[w]_{\ca_{2,\dz}}.
\end{equation}
Note that, for any $t\in \mathbb R$, $e^t+e^{-t}>|t|$. Thus, applying \eqref{eq0723a}, \eqref{eq-0522-18} and \eqref{eq-0522-19}, we find that, for any cube $Q$ of $\rn$,
\begin{align*}
&\frac{1}{{\ch}_\fz^{\delta}(Q)}\int_Q|\ln w(x)-(\ln w)_{Q,\dz}|\,d{\ch}_\fz^{\delta}\\
&\quad\ls \frac{1}{\widetilde{\ch}_\fz^{\delta}(Q)}\int_Q|\ln w(x)-(\ln w)_{Q,\dz}|\,d\widetilde{\ch}_\fz^{\delta}\\
&\quad\ls \frac{1}{\widetilde{\ch}_\fz^{\delta}(Q)}\int_{Q}e^{\ln w(x)-(\ln w)_{Q,\dz}}\,d\widetilde{\ch}_\fz^{\delta}+\frac{1}{\widetilde{\ch}_\fz^{\delta}(Q)}\int_{Q}e^{-[\ln w(x)-(\ln w)_{Q,\dz}]}\,d\widetilde{\ch}_\fz^{\delta}\\
&\quad\ls [w]_{\ca_{2,\dz}},
\end{align*}
which, combined with Proposition \ref{pop-0708-1}, implies that $\ln f\in {\rm {BMO}}(\rn, \ch_{\fz}^{\dz})$ and
$$\|\ln w\|_{{\rm {BMO}}(\rn, \ch_{\fz}^{\dz})}\ls[w]_{\ca_{2,\dz}}.$$
Consequently, \eqref{eq-0628-1} holds true for $p=2$.

When $w\in\ca_{p,\dz}$ with $p\in(1, 2)$, from Lemma \ref{lm0704-1}(ii), we know that $w\in \ca_{2,\dz}$ and $[w]_{\ca_{2,\dz}}\le [w]_{\ca_{p,\dz}}$. Therefore, we have
\begin{align}\label{12e4}
\|\ln w\|_{{\rm {BMO}}(\rn, \ch_{\fz}^{\dz})}\ls[w]_{\ca_{2,\dz}}\ls[w]_{\ca_{p,\dz}},
\end{align}
which means that \eqref{eq-0628-1} also holds true when $w\in \ca_{p,\delta}$ with $p\in(1,2)$.

When $w\in\ca_{p,\dz}$ with $p\in(2, \fz)$, from Lemma \ref{lm0704-1}(iii), we know that
$w^{-\frac{1}{p-1}}\in \ca_{\frac{p}{p-1},\delta}$ with $\frac{p}{p-1}\in(1,2)$ and
$$[w^{-\frac{1}{p-1}}]_{\ca_{\frac{p}{p-1},\dz}}= [w]^{\frac{1}{p-1}}_{\ca_{p,\dz}}.$$
Therefore, form \eqref{12e4}, it follows that
\[\|\ln w\|_{{\rm {BMO}}(\rn, \ch_{\fz}^{\dz})}
=(p-1)\|\ln w^{-\frac{1}{p-1}}\|_{{\rm {BMO}}(\rn, \ch_{\fz}^{\dz})}
\ls [w^{-\frac{1}{p-1}}]_{\ca_{\frac{p}{p-1},\dz}}
\sim [w]^{\frac{1}{p-1}}_{\ca_{p,\dz}}.\]
which implies that, when $w\in\ca_{p,\dz}$ with $p\in(2, \fz)$, \eqref{eq-0628-1} holds true. Combining the above three cases, we obtain (i).

(ii) We show that if $f\in {{\rm {BMO}}}(\rn, \ch_{\fz}^{\dz})$, then \eqref{eq-0628-3} holds true. We begin with the case $p=2$. Indeed, by Remark \ref{12-4}(iv) and Corollary \ref{lm-0522-17}, for any cube $Q$ of $\rn$ and $\gamma\in(0, c)$ with $c$ as in Corollary \ref{lm-0522-17}, we have
\begin{align}\label{eq-0628-x}
&\frac{1}{{\ch}_\fz^{\delta}(Q)}\int_Qe^{\frac{\gamma|f(x)-C_Q|}{\|f\|_{{\rm {BMO}}(\rn, \ch_{\fz}^{\dz})}}}\,d{\ch}_\fz^{\delta}\\
&\quad=1+\frac{1}{{\ch}_\fz^{\delta}(Q)}\int_Q\lf[e^{\frac{\gamma|f(x)-C_Q|}{\|f\|_{{\rm {BMO}}(\rn, \ch_{\fz}^{\dz})}}}-1\r]\,d{\ch}_\fz^{\delta}\noz\\
&\quad=1+\frac{1}{{\ch}_\fz^{\delta}(Q)}\int_{0}^{\fz}{\ch}_\fz^{\delta}\lf(\lf\{x\in Q: e^{\frac{\gamma|f(x)-C_Q|}{\|f\|_{{\rm {BMO}}(\rn, \ch_{\fz}^{\dz})}}}-1>t\r\}\r)\,dt\noz\\
&\quad=1+\frac{1}{{\ch}_\fz^{\delta}(Q)}\int_{0}^{\fz}e^t{\ch}_\fz^{\delta}\lf(\lf\{x\in Q: \frac{\gamma|f(x)-C_Q|}{\|f\|_{{\rm {BMO}}(\rn, \ch_{\fz}^{\dz})}}>t\r\}\r)\,dt\noz\\
&\quad\le 1+C\int_{0}^{\fz}e^{t(1-\frac{c}{\gamma})}\,dt=1+\frac{C\gamma}{c-\gamma}.\noz
\end{align}
Therefore,
\begin{align*}
&\frac{1}{{\ch}_\fz^{\delta}(Q)}\int_Qe^{\frac{\gamma f(x)}{\|f\|_{{{\rm BMO}}(\rn, \ch_{\fz}^{\dz})}}}\,d{\ch}_\fz^{\delta}\frac{1}{{\ch}_\fz^{\delta}(Q)}\int_Qe^{-\frac{\gamma f(x)}{\|f\|_{{\rm {BMO}}(\rn, \ch_{\fz}^{\dz})}}}\,d{\ch}_\fz^{\delta}\\
&=\frac{1}{{\ch}_\fz^{\delta}(Q)}\int_Qe^{\frac{\gamma (f(x)-C_Q)}{\|f\|_{{\rm {BMO}}(\rn, \ch_{\fz}^{\dz})}}}\,d{\ch}_\fz^{\delta}\frac{1}{{\ch}_\fz^{\delta}(Q)}\int_Qe^{-\frac{\gamma (f(x)-C_Q)}{\|f\|_{{\rm {BMO}}(\rn, \ch_{\fz}^{\dz})}}}\,d{\ch}_\fz^{\delta}\\
&\leq\lf(\frac{1}{{\ch}_\fz^{\delta}(Q)}\int_Qe^{\frac{\gamma|f(x)-C_Q|}{\|f\|_{{\rm {BMO}}(\rn, \ch_{\fz}^{\dz})}}}\,d{\ch}_\fz^{\delta}\r)^2\\
&\le \lf(1+\frac{C\gamma}{c-\gamma}\r)^2,
\end{align*}
which implies $$e^{\frac{\gamma f}{\|f\|_{{\rm {BMO}}(\rn, \ch_{\fz}^{\dz})}}}\in \ca_{2,\dz}~~~ \text{and}~~~\lf[e^{\frac{\gamma f}{\|f\|_{{\rm {BMO}}(\rn, \ch_{\fz}^{\dz})}}}\r]_{\ca_{2,\dz}}\le \lf(1+\frac{C\gamma}{c-\gamma}\r)^2.$$ 
Therefore,  \eqref{eq-0628-3} holds true when $p=2$.

Furthermore, Lemma \ref{lm0704-1}(ii) immediately yields that \eqref{eq-0628-3} also holds true when $p\in (2,\fz)$.
Additionally, when $p\in (1,2)$, noting that $p=2(p-1)+1-(p-1)$, by Lemma \ref{lm0704-1}(iv), we have $$e^{\frac{\gamma(p-1) f}{\|f\|_{{\rm {BMO}}(\rn, \ch_{\fz}^{\dz})}}}\in \ca_{p,\dz}~~~ \text{and}~~~\lf[e^{\frac{\gamma(p-1) f}{\|f\|_{{\rm {BMO}}(\rn, \ch_{\fz}^{\dz})}}}\r]_{\ca_{p,\dz}}\le 2 \lf(1+\frac{C\gamma}{c-\gamma}\r)^{2(p-1)}.$$ This proves  \eqref{eq-0628-3} for $p\in (1,2)$. Combining all of these, we obtain (ii), and hence finish the proof of Theorem \ref{tm-0628-1}.
\end{proof}

We now prove Theorem \ref{tm-0720-1}.

\begin{proof}[Proof of Theorem \ref{tm-0720-1}]
(i) We first prove that when $w\in\ca_{1,\dz}$, \eqref{eq-0628-4} holds true. By the definition of $\ca_{1, \dz}$, we have,
for any cube $Q$ of $\rn$ and $\ch_{\fz}^{\dz}$-almost everywhere $y\in Q$,
\begin{equation}\label{eq-720x}
\frac{1}{\ch_{\fz}^{\dz}(Q)}\int_{Q}e^{\ln w(x)}\,d\ch_{\fz}^{\dz}\le [w]_{\ca_{1,\delta}}e^{\ln w(y)}.
\end{equation}
We let
\[E_1\coloneqq  \lf\{x\in \rn: \ln w(x)-\esinf_{y\in Q}\ln w(y)\ge 0\r\}\]
and
\[E_2\coloneqq  \lf\{x\in \rn: \ln w(x)-\esinf_{y\in Q}\ln w(y)<0\r\}.\]
Then $\widetilde{\ch}_\fz^{\delta}(Q\cap E_1)=\widetilde{\ch}_\fz^{\delta}(Q)$ and $\widetilde{\ch}_\fz^{\delta}(Q\cap E_2)=0.$
From this and \eqref{eq-720x}, we deduce that
\begin{equation*}
\frac{1}{\ch_{\fz}^{\dz}(Q)}\int_{Q}e^{\ln w(x)-\esinf_{y\in Q}\ln w(y)}\,d\ch_{\fz}^{\dz}\le [w]_{\ca_{1,\delta}}.
\end{equation*}
Thus, by \eqref{eq0723a} and Remark \ref{rem0704-1}(ii), we further find
\[\frac{1}{\widetilde{\ch}_\fz^{\delta}(Q)}\int_{Q}|\ln w(x)-\esinf_{y\in Q}\ln w(y)|\,d\widetilde{\ch}_\fz^{\delta}\leq\ln [w]_{\ca_{1,\delta}}+ \ln K(n, \dz),\]
where $K(n, \dz)$ is as in \eqref{eq0723a}, which implies that $\ln w\in {\rm BLO}(\rn,\ch_\fz^\delta)$ and \eqref{eq-0628-4} holds true.

In what follows, we show that if $f\in {\rm {BLO}}(\rn, \ch_{\fz}^{\dz})$, then \eqref{eq-0628-5} holds true. Similarly to \eqref{eq-0628-x}, by Remark \ref{12-4}(iv) and Corollary \ref{them-0627-1}, there exists a positive constant $\gamma$ such that, for any cube $Q$ of $\rn$,
\[\frac{1}{{\ch}_\fz^{\delta}(Q)}\int_Qe^{\frac{\gamma [f(x)-\esinf_{y\in Q}f(y)]}{\|f\|_{{\rm {BLO}}(\rn, \ch_{\fz}^{\dz})}}}\,d{\ch}_\fz^{\delta}=\frac{1}{{\ch}_\fz^{\delta}(Q)}\int_Qe^{\frac{\gamma|f(x)-\esinf_{y\in Q}f(y)|}{\|f\|_{{\rm {BLO}}(\rn, \ch_{\fz}^{\dz})}}}\,d{\ch}_\fz^{\delta}\ls 1,\]
which implies
\[\frac{1}{{\ch}_\fz^{\delta}(Q)}\int_Qe^{\frac{\gamma f(x)}{\|f\|_{{\rm {BLO}}(\rn, \ch_{\fz}^{\dz})}}}\,d{\ch}_\fz^{\delta}\ls e^{\frac{\gamma \esinf_{y\in Q}f(y)}{\|f\|_{{\rm {BLO}}(\rn, \ch_{\fz}^{\dz})}}}\sim \esinf_{y\in Q}e^{\frac{\gamma f(y)}{\|f\|_{{\rm {BLO}}(\rn, \ch_{\fz}^{\dz})}}}.\]
Therefore,  \eqref{eq-0628-5} holds true. This completes the proof of Theorem \ref{tm-0720-1}.
\end{proof}

\section{Proofs of Theorems \ref{cor-0721-1} and \ref{cor-0721-2}}\label{s4}

In this section, we prove the decomposition theorems of ${\rm BMO}(\rn,\ch_\fz^\delta)$ and ${\rm BLO}(\rn,\ch_\fz^\delta)$. We begin with the following conclusion, which is of independent interest.

\begin{prop}\label{lm-0722-1}
Let $\dz\in(0, n]$. Then
\[L^{\fz}(\rn, \ch_{\fz}^{\dz})\subsetneqq \rm{BLO}(\rn, \ch_{\fz}^{\dz})\subsetneqq \rm{BMO}(\rn, \ch_{\fz}^{\dz}).\]
\end{prop}

\begin{proof}
We first prove that
\begin{align}\label{eq0722-2}
L^{\fz}(\rn, \ch_{\fz}^{\dz})\subsetneqq \rm{BLO}(\rn, \ch_{\fz}^{\dz}).
\end{align}
Let $f\in L^{\fz}(\rn, \ch_{\fz}^{\dz})$. Then we easily have, for any cube $Q$ of $\rn$,
\begin{align}\label{eq125-6}
\frac{1}{\ch_{\infty}^{\delta}(Q)}\int_{Q}
|f(x)-\esinf_{y\in Q}f(y)|\,d\ch^{\dz}_{\fz}
\leq 2\|f\|_{L^{\fz}(\rn, \ch_{\fz}^{\dz})},
\end{align}
which implies that $L^{\fz}(\rn, \ch_{\fz}^{\dz})\subset \rm{BLO}(\rn, \ch_{\fz}^{\dz})$.
For any $x\in\rn$, let
$$f(x)\coloneqq 
\begin{cases}
-\ln |x|,\quad &\text{when}\ x\neq 0,\\
 0,\quad &\text{when}\ x=0.
\end{cases}
$$
Then, obviously, $f\notin L^{\fz}(\rn, \ch_{\fz}^{\dz})$, but $f\in \rm{BLO}(\rn, \ch_{\fz}^{\dz})$.
Indeed, for any cube $Q\coloneqq Q(x_0, r)$ of $\rn$, where $x_0$ denotes the center of the cube $Q$ and $r$ represents the side length of $Q$, we have $Q\subset B(x_0, \frac{\sqrt{n}}{2}r)$ and
$$\esinf_{y\in B(x_0, \frac{\sqrt{n}}{2}r)}f(y)=-\ln \lf(|x_0|+\frac{\sqrt{n}}{2}r\r).$$
Thus, for any cube $Q$,
\begin{align}\label{eq0722-4}
&\frac{1}{\ch_{\fz}^{\dz}(Q)}\int_{Q}\lf|f(x)-\esinf_{y\in Q}f(y)\r|\,d\ch_{\fz}^{\dz}\\
&\quad\leq \frac{1}{r^{\dz}}\int_{B(x_0, \frac{\sqrt{n}}{2}r)}
\lf[\ln\lf(|x_0|+\frac{\sqrt{n}}{2}r\r)-\ln |x|\r]\,d\ch_{\fz}^{\dz}\noz\\
&\quad=\frac{1}{r^{\dz}}\int_{0}^{\fz}\ch_{\fz}^{\dz}\lf[B\lf(x_0, \frac{\sqrt{n}}{2}r\r)\cap B\lf(0, \lf(|x_0|+\frac{\sqrt{n}}{2}r\r)e^{-t}\r)\r]\,dt=:\rm{I}.\noz
\end{align}
If $|x_0|\leq \frac{3\sqrt{n}}{2}r$, then by \cite[Remark 2.3(i)]{0702-4}, we find that
\begin{align}\label{eq0722-3}
{\rm{I}}\leq \frac{1}{r^{\dz}}\int_{0}^{\fz}\ch_{\fz}^{\dz}\lf(B\lf(0, 2\sqrt{n}re^{-t}\r)\r)\,dt\ls 1.
\end{align}
If $|x_0|> \frac{3\sqrt{n}}{2}r$, then, when $t> \ln (|x_0|+\frac{\sqrt{n}}{2}r)-\ln (|x_0|-\frac{\sqrt{n}}{2}r)$,
$B(x_0, \frac{\sqrt{n}}{2}r)\cap B(0, (|x_0|+\frac{\sqrt{n}}{2}r)e^{-t})=\emptyset$,
which, together with \cite[Remark 2.3(i)]{0702-4} again, further implies that
\begin{align}\label{eq0722-5}
{\rm{I}}&\le \frac{1}{r^{\dz}}\int_{0}^{\ln (|x_0|+\frac{\sqrt{n}}{2}r)-\ln (|x_0|-\frac{\sqrt{n}}{2}r)}\ch_{\fz}^{\dz}\lf(B\lf(x_0, \frac{\sqrt{n}}{2}r\r)\r)\,dt\\
&\le \frac{1}{r^{\dz}}\int_{0}^{\ln 2}\ch_{\fz}^{\dz}\lf(B\lf(x_0, \frac{\sqrt{n}}{2}r\r)\r)\,dt\ls 1.\nonumber
\end{align}
From this, \eqref{eq0722-4} and \eqref{eq0722-3}, we deduce that $f\in \rm{BLO}(\rn, \ch_{\fz}^{\dz})$ and
hence \eqref{eq0722-2} holds true.

Next, we show that
\begin{align}\label{eq0722-6}
\rm{BLO}(\rn, \ch_{\fz}^{\dz})\subsetneqq \rm{BMO}(\rn, \ch_{\fz}^{\dz}).
\end{align}
Obviously, $\rm{BLO}(\rn, \ch_{\fz}^{\dz})\subset \rm{BMO}(\rn, \ch_{\fz}^{\dz})$.
Let, for any $x\in\rn$,
$$g(x)\coloneqq 
\begin{cases}
\ln |x|,\quad &\text{when}\ x\neq 0,\\
0,\quad &\text{when}\ x=0.
\end{cases}
$$
Then, for any cube $Q$ of $\rn$ containing the origin, $\esinf_{y\in Q}g(y)=-\fz$, which implies that
$g\notin \rm{BLO}(\rn, \ch_{\fz}^{\dz})$. However, $g\in \rm{BMO}(\rn, \ch_{\fz}^{\dz})$. Indeed,
for any cube $Q$ of $\rn$, by \eqref{eq0722-4}, \eqref{eq0722-3} and \eqref{eq0722-5}, we know that
\begin{align*}
\inf_{c\in \mathbb R}\frac{1}{\ch_{\fz}^{\dz}(Q)}\int_{Q}\lf|g(x)-c\r|\,d\ch_{\fz}^{\dz}
&\leq \frac{1}{\ch_{\fz}^{\dz}(Q)}\int_{Q}\lf|g(x)-\esup_{y\in Q}g(y)\r|\,d\ch_{\fz}^{\dz}\\
&=\frac{1}{\ch_{\fz}^{\dz}(Q)}\int_{Q}\lf|f(x)-\esinf_{y\in Q}f(y)\r|\,d\ch_{\fz}^{\dz}\ls 1,\noz
\end{align*}
which implies that $g\in\rm{BMO}(\rn, \ch_{\fz}^{\dz})$ and hence \eqref{eq0722-6} holds true.
This finishes the proof of Proposition \ref{lm-0722-1}.
\end{proof}

To show Theorems \ref{cor-0721-1} and \ref{cor-0721-2}, we first show the following lemma.

\begin{lem}\label{lm0721-1}
Let $\dz\in (0, n]$, $p,q\in[1,\fz)$ and $r\in (0, q)$. If $T$ is a bounded operator from
$L^{p}(\rn,\ch_\fz^\delta)$ to $L^{q,\infty}(\rn,\ch_\fz^\delta)$, i.e.,
there exists a positive constant $C$ such that, for any $\lz\in(0,\fz)$
and $f\in L^{p}(\rn,\ch_\fz^\delta)$,
\begin{align}\label{eq-0721-1}
\lz \lf[\ch_{\fz}^{\dz}\lf(\lf\{x\in\rn:\ |Tf(x)|>\lambda\r\}\r)\r]^\frac1q
\leq C\|f\|_{L^p(\rn, \ch_{\fz}^{\dz})}.
\end{align}
Then, for any subset $E\subset \rn$ with $\ch_{\fz}^{\dz}(E)< \fz$, we have
\begin{align}\label{eq-0721-2}
\lf(\int_E|Tf(x)|^r\,d\ch_{\fz}^{\dz}\r)^{\frac{1}{r}}\leq 2C\lf(\frac{q}{q-r}\r)^{\frac{1}{q}}\ch_{\fz}^{\dz}(E)^{\frac{1}{r}-\frac{1}{q}}\|f\|_{L^{p}(\rn,\ch_\fz^\delta)}.
\end{align}
\end{lem}

\begin{proof}
Without loss of generality, we may assume $\|f\|_{L^{p}(\rn,\ch_\fz^\delta)}\neq0$. For any $M>0$, by \eqref{eq-0721-1},
we find that
\begin{align*}
\int_E|Tf(x)|^r\,d\ch_{\fz}^{\dz}
&=r\int_0^{\fz}\lambda^{r-1}\ch_{\fz}^{\dz}\lf(\lf\{x\in E:\ |Tf(x)|>\lambda\r\}\r)\,d\lambda\\
&\leq r\int_0^{M}\lambda^{r-1}\ch_{\fz}^{\dz}\lf(E\r)\,d\lambda+r\int_M^{\fz}\lambda^{r-1}\lf(\frac{C}{\lambda}\|f\|_{L^p(\rn, \ch_{\fz}^{\dz})}\r)^{q}\,d\lambda\\
&\leq
M^r\ch_{\fz}^{\dz}\lf(E\r)+\frac{rC^q}{q-r}\|f\|^q_{L^{p}(\rn,\ch_\fz^\delta)}M^{r-q}.
\end{align*}
By setting $M\coloneqq C(\frac{r}{q-r})^{1/q}\ch_{\fz}^{\dz}\lf(E\r)^{-\frac{1}{q}}\|f\|_{L^{p}(\rn,\ch_\fz^\delta)}$,
we immediately obtain \eqref{eq-0721-2}. This finishes the proof of Lemma \ref{lm0721-1}.
\end{proof}

\begin{rem}\label{rem-0725}
Let $\delta\in(0,n]$ and $r\in (0,p)$. Then, by \cite[Theorem 1.2]{COS2024}, we find that the operator $\cm_{\ch_\fz^\delta}$
satisfies the assumptions of Lemma \ref{lm0721-1} $p=q\in[1,\fz)$. Therefore, for any subset $E\subset \rn$ with $\ch_{\fz}^{\dz}(E)< \fz$ and $f\in L^{p}(\rn,\ch_\fz^\delta)$, we have
\begin{align*}
\lf(\int_E|\cm_{\ch_\fz^\delta}f(x)|^r\,d\ch_{\fz}^{\dz}\r)^{\frac{1}{r}}\ls \lf(\frac{p}{p-r}\r)^{\frac{1}{p}}\ch_{\fz}^{\dz}(E)^{\frac{1}{r}-\frac{1}{p}}\|f\|_{L^{p}(\rn,\ch_\fz^\delta)}.
\end{align*}
\end{rem}

Applying Lemma \ref{lm0721-1}, we establish the following result.

\begin{lem}\label{lm0721-5}
Let $\delta\in(0,n]$.
\begin{enumerate}
\item[{\rm(i)}]
Let $f\in L^1_{\loc}(\rn,\ch_{\fz}^{\dz})$. If $\cm_{\ch_{\fz}^{\dz}}f(x)<\fz$ for $\ch_\fz^\delta$-almost every $x\in\rn$, then, for any $\alpha\in(0,1)$,
$(\cm_{\ch_{\fz}^{\dz}}f)^{\alpha}\in \ca_{1,\dz}$ and there exists a positive constant $C$ that depends only on $n$ and $\dz$, such that 
\begin{align}\label{eq-26-16}
\lf[\lf(\cm_{\ch_{\fz}^{\dz}}f\r)^{\alpha}\r]_{\ca_{1,\dz}}\le C\lf[\lf(\frac{1}{1-\alpha}\r)^{\alpha}+1\r].
\end{align}
\item[{\rm(ii)}]
Let $w\in \ca_{1,\dz}$. Then there exist a function
$f\in L^1_{\loc}(\rn, \ch_{\fz}^{\dz})$ and a  function $b$ such that
\begin{enumerate}
\item[{\rm(A)}] for $\ch_\fz^\delta$-almost every $x\in\rn$, $\cm_{\ch_{\fz}^{\dz}}f(x)<\fz$;
\item[{\rm(B)}] for $\ch_\fz^\delta$-almost every $x\in\rn$, $0<C_1\le b(x)\le C_2$ with constants $C_1$ and $C_2$;
\item[{\rm(C)}] for some $\alpha\in (0,1)$ depending only on $n$, $\dz$ and $[w]_{\ca_{1,\dz}}$, $w=b(\cm_{\ch_{\fz}^{\dz}}f)^{\alpha}$.
\end{enumerate}
Moreover, there exists a positive constant $C=C(n, \dz, [w]_{\ca_{1,\dz}})$ such that
\[\|\ln b\|_{L^{\fz}(\rn,\ch_{\fz}^{\dz})}\leq C.\]
\end{enumerate}

\end{lem}

\begin{proof}
We first prove (i). It is enough to show that, for any cube $Q_0$,
\begin{align}\label{eq-26-16-1}
\frac{1}{\ch_{\fz}^{\dz}\lf(Q_0\r)}\int_{Q_0}\lf[\cm_{\ch_{\fz}^{\dz}}f(y)\r]^{\alpha}\,d\ch_{\fz}^{\dz}\le C\lf[\cm_{\ch_{\fz}^{\dz}}f(x)\r]^{\alpha}~\text{for}~\ch_{\fz}^{\dz}-\text{almost~every}~x\in Q_0,
\end{align}
where $C$ is a positive constant independent of $Q_0$ and $x$. To this end, for any given cube $Q_0$, for any $x\in Q_0$, let
\[E_{x,1}\coloneqq \{\ {\rm cube}\  Q:\ \ch_{\fz}^{\dz}\lf(Q\r)\le \ch_{\fz}^{\dz}\lf(2Q_0\r)\ {\rm and} \ Q\ni x\}\]
and
\[E_{x,2}\coloneqq \{\ {\rm cube}\  Q:\ \ch_{\fz}^{\dz}\lf(Q\r)> \ch_{\fz}^{\dz}\lf(2Q_0\r)\ {\rm and} \ Q\ni x\}.\]
Clearly, for any $x\in Q_0$,
\begin{align*}
\cm_{\ch_{\fz}^{\dz}}f(x)\leq \sup_{Q\in E_{x,1}}\frac{1}{\ch_{\fz}^{\dz}\lf(Q\r)}\int_{Q}|f(y)|\,d\ch_{\fz}^{\dz}
+\sup_{Q\in E_{x,2}}\frac{1}{\ch_{\fz}^{\dz}\lf(Q\r)}\int_{Q}|f(y)|\,d\ch_{\fz}^{\dz}=:F_1(x)+F_2(x).
\end{align*}
Since, for any cube $Q\in E_{x,1}$, $Q\subset 5Q_0$, it follows that
\begin{equation}\label{eq0721-5}
F_1(x)=\sup_{Q\in E_{x,1}}\frac{1}{\ch_{\fz}^{\dz}\lf(Q\r)}\int_{Q}|f(y)|\mathbf{1}_{5Q_0}(y)\,d\ch_{\fz}^{\dz}\leq \cm_{\ch_{\fz}^{\dz}}(f\mathbf{1}_{5Q_0})(x).
\end{equation}
For $F_2(x)$, notice that, for any cube $Q\in E_{x,2}$, we have $Q_0\subset 4Q$. Thus,
\begin{align*}
\frac{1}{\ch_{\fz}^{\dz}\lf(Q\r)}\int_{Q}|f(y)|\,d\ch_{\fz}^{\dz}=
\frac{4^{\dz}}{\ch_{\fz}^{\dz}\lf(4Q\r)}\int_{Q}|f(y)|\,d\ch_{\fz}^{\dz}\leq 4^{\dz}\inf_{x\in 4Q}\cm_{\ch_{\fz}^{\dz}}f(x)\leq 4^{\dz}\inf_{x\in Q_0}\cm_{\ch_{\fz}^{\dz}}f(x),
\end{align*}
which implies that
\begin{align*}
F_2(x)\leq 4^{\dz}\inf_{x\in Q_0}\cm_{\ch_{\fz}^{\dz}}f(x).
\end{align*}
By this, the fact $\alpha\in(0,1)$, \eqref{eq0721-5} and Remark \ref{rem-0725} with $E=Q_0$, $r=\alpha$ and $p=1$, we conclude that, for any $x\in Q_0$,
\begin{align*}
&\frac{1}{\ch_{\fz}^{\dz}\lf(Q_0\r)}\int_{Q_0}\lf[\cm_{\ch_{\fz}^{\dz}}f(y)\r]^{\alpha}\,d\ch_{\fz}^{\dz}\\
&\quad \ls \frac{1}{\ch_{\fz}^{\dz}\lf(Q_0\r)}\int_{Q_0}\lf[\cm_{\ch_{\fz}^{\dz}}(f\mathbf{1}_{5Q_0})(y)\r]^{\alpha}\,d\ch_{\fz}^{\dz}
+\frac{1}{\ch_{\fz}^{\dz}\lf(Q_0\r)}\int_{Q_0}\lf[\inf_{y\in Q_0}\cm_{\ch_{\fz}^{\dz}}f(y)\r]^{\alpha}\,d\ch_{\fz}^{\dz}\\
&\quad \ls \lf(\frac{1}{1-\alpha}\r)^{\alpha}\lf(\frac{1}{\ch_{\fz}^{\dz}\lf(Q_0\r)}\int_{5Q_0}|f(y)|\,d\ch_{\fz}^{\dz}\r)^{\alpha}+\lf[\inf_{y\in Q_0}\cm_{\ch_{\fz}^{\dz}}f(y)\r]^{\alpha}\\
&\quad\ls \lf[\lf(\frac{1}{1-\alpha}\r)^{\alpha}+1\r]\lf[\inf_{y\in Q_0}\cm_{\ch_{\fz}^{\dz}}f(y)\r]^{\alpha}\ls \lf[\lf(\frac{1}{1-\alpha}\r)^{\alpha}+1\r]\lf[\cm_{\ch_{\fz}^{\dz}}f(x)\r]^{\alpha},
\end{align*}
which implies \eqref{eq-26-16-1}
and hence $(\cm_{\ch_{\fz}^{\dz}}f)^\az \in \ca_{1, \dz}$ and its $\ca_{1, \dz}$ constant satisfies \eqref{eq-26-16}.

For (ii), according to Lemma \ref{lm0704-1}(vi), there exists a constant $\gamma\in(0,1)$ such that $w^{\gamma+1}\in\ca_{1,\dz}$. Thus, for $\ch_\fz^\delta$-almost every $x\in \rn$,
$\cm_{\ch_{\fz}^{\dz}}(w^{\gamma+1})(x)<\fz$.
Let $\alpha\coloneqq \frac{1}{1+\gamma}$ and, for any $x\in \rn$,
 $$f(x)\coloneqq w(x)^{1/\az},\ b(x)\coloneqq w(x)[\cm_{\ch_{\fz}^{\dz}}(w^{1/\az})(x)]^{-\az}.$$
Then $w=b[\cm_{\ch_{\fz}^{\dz}}(f)]^\az$, and hence (A) and (C) are proved.
Moreover, by \cite[Corollary 4.4]{0702-4}, \eqref{eq0723a} and H\"older's inequality, we conclude that,
for $\ch_\fz^\delta$-almost every $x\in\rn$,
\[w(x)\ls \cm_{\ch_{\fz}^{\dz}}w(x)\ls [\cm_{\ch_{\fz}^{\dz}}(w^{1+\gamma})(x)]^{\frac{1}{1+\gamma}},\]
which implies that $b(x)\ls 1$. Additionally, by $w^{\gamma+1}\in\ca_{1,\dz}$, we have
\[[\cm_{\ch_{\fz}^{\dz}}(w^{1+\gamma})(x)]^{\frac{1}{1+\gamma}}\ls w(x),\]
which implies that $b(x)\gs 1$. Therefore, (B) holds true and $\ln b\in L^{\fz}({\rn,\ch_{\fz}^{\dz}})$.
This finishes the proof of Lemma \ref{lm0721-5}.
\end{proof}

By the definition of $\rm{BMO}(\rn,\ch_{\fz}^{\dz})$, we immediately obtain the following properties.

\begin{lem}\label{rem125-1}
Let $\dz\in (0,n]$.
\begin{enumerate}
\item[{\rm(i)}]
If $f\in {\rm{BMO}}(\rn,\ch_{\fz}^{\dz})$. Then, for any $h\in \mathbb R$,  $hf\in {\rm{BMO}}(\rn,\ch_{\fz}^{\dz})$ and
\[\|hf\|_{{\rm{BMO}}(\rn,\ch_{\fz}^{\dz})}=|h|\|f\|_{{\rm{BMO}}(\rn,\ch_{\fz}^{\dz})}.\]
\item[{\rm(ii)}]
If $f, g\in {\rm{BMO}}(\rn,\ch_{\fz}^{\dz})$. Then, $f+g\in {\rm{BMO}}(\rn,\ch_{\fz}^{\dz})$ and
\[\|f+g\|_{{\rm{BMO}}(\rn,\ch_{\fz}^{\dz})}\leq 2\lf(\|f\|_{{\rm{BMO}}(\rn,\ch_{\fz}^{\dz})}+\|g\|_{{\rm{BMO}}(\rn,\ch_{\fz}^{\dz})}\r).\]
\end{enumerate}
\end{lem}

We now prove Theorem \ref{cor-0721-1}.

\begin{proof}[Proof of Theorem \ref{cor-0721-1}]
(i) For $f\in {\rm{BMO}}(\rn,\ch_{\fz}^{\dz})$, then Theorem \ref{tm-0628-1}(ii) yields 
\[e^{\frac{c_1 f}{\|f\|_{{\rm{BMO}}(\rn,\ch_{\fz}^{\dz})}}}\in\ca_{2,\dz}\quad \text{and}\quad \lf[e^{\frac{c_1 f}{\|f\|_{{\rm{BMO}}(\rn,\ch_{\fz}^{\dz})}}}\r]_{\ca_{2,\dz}}\leq c_2,\]
where $c_1$ and $c_2$ are positive constants depending only $n$ and $\dz$.
Subsequently, by Jones' factorization theorem established in \cite[Theorem 1.10]{0702-4}, there exist two weights $w_1, w_2\in\ca_{1,\dz}$ such that
\begin{align}\label{eq124-a}
e^{\frac{c_1 f}{\|f\|_{{\rm{BMO}}(\rn,\ch_{\fz}^{\dz})}}}=w_1w_2^{-1},
\end{align}
where the $\ca_{1,\dz}$ constants $[w_1]_{\ca_{1,\dz}}$ and $[w_2]_{\ca_{1,\dz}}$ are bounded above by a constant depending only on $n$ and $\dz$. Applying Lemma \ref{lm0721-5}(ii) to $w_1$, we obtain functions $g_1\in L^{1}_{\loc}(\rn,\ch_{\fz}^{\dz})$ and $b_1$ satisfying $\ln b_1\in L^{\fz}(\rn,\ch_{\fz}^{\dz})$, and $\alpha_1\in(0,1)$
depending only on $n$ and $\dz$ such that
\begin{align}\label{eq124q}
w_1=b_1\lf(\cm_{\ch_{\fz}^{\dz}}g_1\r)^{\alpha_1}.
\end{align}
Similarly, using Lemma \ref{lm0721-5}(ii) to $w_2$, there exist functions $g_2\in L^{1}_{\loc}(\rn,\ch_{\fz}^{\dz})$ and $b_2$ satisfying $\ln b_2\in L^{\fz}(\rn,\ch_{\fz}^{\dz})$, and $\alpha_2\in(0,1)$ depending only on $n$ and $\dz$ such that
\begin{align}\label{eq124w}
w_2=b_2\lf(\cm_{\ch_{\fz}^{\dz}}g_2\r)^{\alpha_2}.
\end{align}
Combining \eqref{eq124-a}, \eqref{eq124q} and \eqref{eq124w}, we immediately deduce
\[f=\frac{\|f\|_{{\rm{BMO}}(\rn,\ch_{\fz}^{\dz})}}{c_1}\lf[\alpha_1\ln \cm_{\ch_{\fz}^{\dz}}g_1-\alpha_2\ln \cm_{\ch_{\fz}^{\dz}}g_2+\ln b_1-\ln b_2\r].\]
Now set
\[\alpha=\frac{\alpha_1}{c_1}\|f\|_{{\rm{BMO}}(\rn,\ch_{\fz}^{\dz})},\quad \beta=\frac{\alpha_2}{c_1}\|f\|_{{\rm{BMO}}(\rn,\ch_{\fz}^{\dz})},\quad b=\frac{\ln b_1-\ln b_2}{c_1}\|f\|_{{\rm{BMO}}(\rn,\ch_{\fz}^{\dz})}.\]
Then
\[f=\alpha\ln \cm_{\ch_{\fz}^{\dz}}g_1-\beta\ln \cm_{\ch_{\fz}^{\dz}}g_2+b,\]
and one has the estimate
\[\alpha+\beta+\|b\|_{L^{\fz}(\rn,\ch_{\fz}^{\dz})}\leq \frac{1}{c_1}\lf(\alpha_1+\alpha_2+\|\ln b_1\|_{L^{\fz}(\rn,\ch_{\fz}^{\dz})}+\|\ln b_2\|_{L^{\fz}(\rn,\ch_{\fz}^{\dz})}\r)\|f\|_{{\rm{BMO}}(\rn,\ch_{\fz}^{\dz})}.\]

(ii) Define
\[h=2\lf(\alpha+\beta+\|b\|_{L^{\fz}(\rn,\ch_{\fz}^{\dz})}\r).\]
Then $\frac{\alpha}{h}\leq \frac12$ and $\frac{\beta}{h}\leq \frac12$. Moreover,
\begin{align}\label{eq125-1}
\lf\|\frac{b}{h}\r\|_{{\rm{BMO}}(\rn, \ch_{\fz}^{\dz})}=\frac{1}{h}\|b\|_{{\rm{BMO}}(\rn, \ch_{\fz}^{\dz})}\leq \frac{1}{h}\|b\|_{L^{\fz}(\rn,\ch_{\fz}^{\dz})}\leq \frac{1}{2}.
\end{align}
By Lemma \ref{lm0721-5}(i), we know that both $(\cm_{\ch_{\fz}^{\dz}}g_1)^{\frac{\alpha}{h}}$ and $(\cm_{\ch_{\fz}^{\dz}}g_2)^{\frac{\beta}{h}}$ belong to $\ca_{1,\dz}$, and their $\ca_{1,\dz}$ constants satisfy \eqref{eq-26-16} (admit uniform upper bound depending only $n$ and $\dz$). Hence, \cite[Theorem 1.10]{0702-4} implies that
\[\lf(\cm_{\ch_{\fz}^{\dz}}g_1\r)^{\frac{\alpha}{h}}\lf(\cm_{\ch_{\fz}^{\dz}}g_2\r)^{-\frac{\beta}{h}}\in\ca_{2,\dz},\]
with the upper bound of $\ca_{2,\dz}$ constant depending only on $n$ and $\dz$. Applying Theorem \ref{tm-0628-1}(i), we obtain $\frac{\alpha}{h}\ln(\cm_{\ch_{\fz}^{\dz}}g_1)-\frac{\beta}{h}\ln(\cm_{\ch_{\fz}^{\dz}}g_2)\in{\rm BMO}(\rn,\ch_{\fz}^{\dz})$ and there exists a constant $C=C(n,\dz)$ such that
\[\lf\|\frac{\alpha}{h}\ln(\cm_{\ch_{\fz}^{\dz}}g_1)-\frac{\beta}{h}\ln(\cm_{\ch_{\fz}^{\dz}}g_2)\r\|_{{\rm BMO}(\rn,\ch_{\fz}^{\dz})}\leq C.\]
By this, \eqref{eq125-1} and Lemma \ref{rem125-1}, we conclude that $f\in {\rm{BMO}}(\rn,\ch_{\fz}^{\dz})$ and
\begin{align*}
\|f\|_{{\rm{BMO}}(\rn,\ch_{\fz}^{\dz})}&=h\lf\|\frac{\alpha}{h}\ln(\cm_{\ch_{\fz}^{\dz}}g_1)-\frac{\beta}{h}\ln(\cm_{\ch_{\fz}^{\dz}}g_2)+\frac{b}{h}\r\|_{{\rm BMO}(\rn,\ch_{\fz}^{\dz})}\\
&\leq (4C+2)\lf(\alpha+\beta+\|b\|_{L^{\fz}(\rn,\ch_{\fz}^{\dz})}\r),
\end{align*}
which completes the proof of Theorem \ref{cor-0721-1}.
\end{proof}

To show Theorem \ref{cor-0721-2}, we also need the following lemma.

\begin{lem}\label{rem125-2}
Let $\dz\in (0,n]$.
\begin{enumerate}
\item[{\rm(i)}]
If $f\in {\rm{BLO}}(\rn,\ch_{\fz}^{\dz})$. Then, for any $h\in [0,\fz)$, we have, $hf\in {\rm{BLO}}(\rn,\ch_{\fz}^{\dz})$ and
\begin{align}\label{eq125-2}
\|hf\|_{{\rm{BLO}}(\rn,\ch_{\fz}^{\dz})}=h\|f\|_{{\rm{BLO}}(\rn,\ch_{\fz}^{\dz})}.
\end{align}
However, when $h<0$, $hf$ does not necessarily belong to ${\rm{BLO}}(\rn,\ch_{\fz}^{\dz})$.
\item[{\rm(ii)}]
If $f, g\in {\rm{BLO}}(\rn,\ch_{\fz}^{\dz})$. Then, $f+g\in {\rm{BLO}}(\rn,\ch_{\fz}^{\dz})$ and
\begin{align}\label{eq125-3}
\|f+g\|_{{\rm{BLO}}(\rn,\ch_{\fz}^{\dz})}\leq 2\lf(\|f\|_{{\rm{BLO}}(\rn,\ch_{\fz}^{\dz})}+\|g\|_{{\rm{BLO}}(\rn,\ch_{\fz}^{\dz})}\r).
\end{align}
However, $f-g$ does not necessarily belong to ${\rm{BLO}}(\rn,\ch_{\fz}^{\dz})$.
\item[{\rm(iii)}]
If $f\in {\rm{BLO}}(\rn,\ch_{\fz}^{\dz})$ and $g\in L^{\fz}(\rn,\ch_{\fz}^{\dz})$. Then, for any $h\in\mathbb R$, we have $f+hg\in {\rm{BLO}}(\rn,\ch_{\fz}^{\dz})$ and
\[\|f+hg\|_{{\rm{BLO}}(\rn,\ch_{\fz}^{\dz})}\leq 2\lf(\|f\|_{{\rm{BLO}}(\rn,\ch_{\fz}^{\dz})}+2|h|\|g\|_{L^{\fz}(\rn,\ch_{\fz}^{\dz})}\r).\]
\end{enumerate}
\end{lem}

\begin{proof}
(i)
For any cube $Q$ of $\rn$, it is not difficult to see that
\[\esinf_{y\in Q}(hf)(y)=h\esinf_{y\in Q}f(y).\]
Therefore, from the definition of ${\rm{BLO}}(\rn, \ch_{\fz}^{\dz})$, we immediately obtain \eqref{eq125-2}. Furthermore, from the proof of Proposition \ref{lm-0722-1}, we find that $f\in {\rm{BLO}}(\rn, \ch_{\fz}^{\dz})$. However, $g=-f\notin {\rm{BLO}}(\rn, \ch_{\fz}^{\dz})$.

(ii) For any cube $Q$ of $\rn$, using this and \eqref{eq125-0}, we have
\begin{align*}
&\frac{1}{\ch_{\infty}^{\delta}(Q)}\int_{Q}
|f(x)+g(x)-\esinf_{y\in Q}[f(y)+g(y)]|\,d\ch^{\dz}_{\fz}\\
&\quad\leq \frac{1}{\ch_{\infty}^{\delta}(Q)}\int_{Q}
|f(x)+g(x)-\esinf_{y\in Q}f(y)-\esinf_{y\in Q}g(y)|\,d\ch^{\dz}_{\fz}\\
&\quad\leq \frac{2}{\ch_{\infty}^{\delta}(Q)}\int_{Q}
|f(x)-\esinf_{y\in Q}f(y)|\,d\ch^{\dz}_{\fz}+\frac{2}{\ch_{\infty}^{\delta}(Q)}\int_{Q}
|g(x)-\esinf_{y\in Q}g(y)|\,d\ch^{\dz}_{\fz},
\end{align*}
which implies that \eqref{eq125-3} holds and $f+g\in {\rm{BLO}}(\rn,\ch_{\fz}^{\dz})$. Additionally, if we choose $f\equiv 0$, then when $g\in {\rm{BLO}}(\rn,\ch_{\fz}^{\dz})$, by (i), we know that $f-g$ does not necessarily belong to ${\rm{BLO}}(\rn,\ch_{\fz}^{\dz})$.

(iii) By \eqref{eq125-6}, we immediately have that $hg\in {\rm{BLO}}(\rn,\ch_{\fz}^{\dz})$ and
\[\|hg\|_{{\rm{BLO}}(\rn,\ch_{\fz}^{\dz})}\leq 2|h|\|g\|_{L^{\fz}(\rn, \ch_{\fz}^{\dz})}.\]
Then, using this and (ii), we immediately obtain (iii).
\end{proof}

\begin{proof}[Proof of Theorem \ref{cor-0721-2}]
By Theorem \ref{tm-0720-1}, Lemmas \ref{lm0721-5} and \ref{rem125-2}, and an argument similar to that used in the proof of  Theorem \ref{cor-0721-1}, we obtain Theorem \ref{cor-0721-2}; the details are omitted.
\end{proof}

\section{Proofs of Theorems \ref{them-0629-1} and \ref{them-127-2}}\label{s5}

We first prove Theorem \ref{them-0629-1}.

\begin{proof}[Proof of Theorem \ref{them-0629-1}]
We complete this proof by two steps.

\emph{\textbf{Step 1)}} We prove that, for any $w\in\ca_{p,\dz}$ with $p\in[1,\fz)$,
\begin{equation}\label{eq-0629-2}
\|f\|_{{\rm {BMO}}(\rn, \ch_{\fz}^{\dz})}\sim\|f\|_{{{\rm {BMO}}}_{w}^p(\rn, \ch_{\fz}^{\dz})}.
\end{equation}
To this end, we first show that, for any $f\in {\rm BMO}(\rn, \ch_{\fz}^{\dz})$,
\begin{equation}\label{eq-0629-3}
\|f\|_{{{\rm {BMO}}}_{w}^p(\rn, \ch_{\fz}^{\dz})}\ls \|f\|_{{\rm {BMO}}(\rn, \ch_{\fz}^{\dz})}.
\end{equation}
Since $f\in {\rm {BMO}}(\rn, \ch_{\fz}^{\dz})$, according to Corollary \ref{lm-0522-17},
there exists positive constants $c, C>0$ such that for any $t>0$ and any cube $Q$ of $\rn$,
\[\ch_{\fz}^{\dz}\lf(\lf\{x\in Q: |f(x)-C_Q|>t\r\}\r)\le C\ch_{\fz}^{\dz}(Q)e^{-\frac{ct}{\|f\|_{{\rm {BMO}}(\rn, \ch_{\fz}^{\dz})}}},\]
where $C_Q$ is as in Corollary \ref{lm-0522-17}. Thus, for any $w\in\ca_{p,\dz}$, by Lemma \ref{lm0704-1}(v), there exists a positive constant $\varepsilon$ such that, for any cube $Q$ of $\rn$,
\[w_{\ch_{\fz}^{\dz}}\lf(\lf\{x\in Q: |f(x)-C_Q|>t\r\}\r)\ls w_{\ch_{\fz}^{\dz}}(Q)e^{-\frac{c\varepsilon t}{\|f\|_{{\rm {BMO}}(\rn, \ch_{\fz}^{\dz})}}}.\]
From this and Lemma \ref{lm-0919-4}, we deduce that, for any cube $Q$ of $\rn$,
\begin{align*}
\int_Q|f(x)-C_Q|^pw(x)\,d\ch_{\fz}^{\dz}
&\le 4p\int_0^{\fz}t^{p-1}w_{\ch_{\fz}^{\dz}}\lf(\lf\{x\in Q: |f(x)-C_Q|>t\r\}\r)\,dt\\
&\ls w_{\ch_{\fz}^{\dz}}(Q)\int_0^{\fz}t^{p-1}e^{-\frac{c\varepsilon t}{\|f\|_{{\rm {BMO}}(\rn, \ch_{\fz}^{\dz})}}}\,dt \noz\\
&\sim w_{\ch_{\fz}^{\dz}}(Q)\|f\|^p_{{\rm {BMO}}(\rn, \ch_{\fz}^{\dz})}\int_0^{\fz}t^{p-1}e^{-c\varepsilon t}\,dt, \noz
\end{align*}
which implies \eqref{eq-0629-3}.

Next, we verify that, for any $f\in {\rm BMO}_w^p(\rn,\ch_\fz^\delta)$,
\begin{equation}\label{eq-0629-4}
\|f\|_{{{\rm {BMO}}}(\rn, \ch_{\fz}^{\dz})} \le 2[w]^{1/p}_{\ca_{p,\dz}} \|f\|_{{{\rm {BMO}}}_{w}^p(\rn, \ch_{\fz}^{\dz})}.
\end{equation}
If $p>1$, by H\"{o}lder's inequality \eqref{12e1}, we know that, for any $c\in\rr$,
\begin{align*}
\int_Q|f(x)-c|\,d\ch_{\fz}^{\dz}
\le 2\lf(\int_Q|f(x)-c|^pw(x)\,d\ch_{\fz}^{\dz}\r)^{\frac{1}{p}}\lf(\int_Qw(x)^{-\frac{1}{p-1}}\,d\ch_{\fz}^{\dz}\r)^{1-\frac{1}{p}},
\end{align*}
which further implies
\begin{align*}
\|f\|_{{{\rm {BMO}}}(\rn, \ch_{\fz}^{\dz})}
&\le 2\|f\|_{{{\rm {BMO}}}_{w}^p(\rn, \ch_{\fz}^{\dz})} \lf(\frac1{\ch_{\fz}^{\dz}(Q)}\int_Qw(x)\,d\ch_{\fz}^{\dz}\r)^{\frac{1}{p}}\lf(\frac1{\ch_{\fz}^{\dz}(Q)}\int_Qw(x)^{-\frac{1}{p-1}}\,d\ch_{\fz}^{\dz}\r)^{1-\frac{1}{p}}\noz\\
&\le 2[w]^{1/p}_{\ca_{p,\dz}} \|f\|_{{{\rm {BMO}}}_{w}^p(\rn, \ch_{\fz}^{\dz})}.\noz
\end{align*}
If $p=1$, using $w\in\ca_{1,\dz}$, we then obtain, for any $c\in\rr$,
\begin{align*}
\int_Q|f(x)-c|\,d\ch_{\fz}^{\dz}
&\le\int_Q|f(x)-c|w(x)\,d\ch_{\fz}^{\dz}\|w^{-1}\|_{L^{\fz}(Q,\ch_{\fz}^{\dz})}.
\end{align*}
Hence
\begin{align*}
\|f\|_{{{\rm {BMO}}}(\rn, \ch_{\fz}^{\dz})}
\le \|f\|_{{{\rm {BMO}}}_{w}^1(\rn, \ch_{\fz}^{\dz})}\frac{w_{\ch_{\fz}^{\dz}}(Q)}{ \ch_{\fz}^{\dz}(Q)}\|w^{-1}\|_{L^{\fz}(Q,\ch_{\fz}^{\dz})}
\le [w]_{\ca_{1,\dz}}\|f\|_{{{\rm {BMO}}}_{w}^1(\rn, \ch_{\fz}^{\dz})}.\noz
\end{align*}
Therefore, \eqref{eq-0629-4} and hence \eqref{eq-0629-2} holds true.

\emph{\textbf{Step 2)}} We show that, for any $q\in(0,\fz)$ and $w\in\ca_{p,\dz}$ with $p\in[1,\fz)$,
\begin{equation}\label{eq-0629-1}
\|f\|_{{{\rm {BMO}}}_{w}^q(\rn, \ch_{\fz}^{\dz})}\sim \|f\|_{{{\rm {BMO}}}^p_w(\rn, \ch_{\fz}^{\dz})}.
\end{equation}
Let $f\in {{\rm {BMO}}}^q_w(\rn, \ch_{\fz}^{\dz})$. Then, by Proposition \ref{tm-0523-2}, there exists positive constants $c, C>0$ such that for any $t>0$ and any cube $Q$ of $\rn$
\begin{equation*}
w_{\ch_{\fz}^{\dz}}(\{x\in Q:\ |f(x)-C_{f,w,Q}|>t\})\le Cw_{\ch_{\fz}^{\dz}}(Q)e^{-\frac{ct}{\|f\|_{{{\rm {BMO}}}^q_w(\rn, \ch_{\fz}^{\dz})}}},
\end{equation*}
where
$C_{f,w,Q}$ is as in \eqref{gamma}. Therefore,
\begin{align*}
\int_Q|f(x)-C_{f,w,Q}|^pw(x)\,d\ch_{\fz}^{\dz}
&\le 4p\int_0^{\fz}t^{p-1}w_{\ch_{\fz}^{\dz}}\lf(\lf\{x\in Q:\ |f(x)-C_{f,w,Q}|>t\r\}\r)\,dt\\
&\ls w_{\ch_{\fz}^{\dz}}(Q)\int_0^{\fz}t^{p-1}e^{-\frac{ct}{\|f\|_{{{\rm {BMO}}}^q_{w}(\rn, \ch_{\fz}^{\dz})}}}\,dt \noz\\
&\sim w_{\ch_{\fz}^{\dz}}(Q)\|f\|^p_{{{\rm {BMO}}}_w^q(\rn, \ch_{\fz}^{\dz})}\int_0^{\fz}t^{p-1}e^{-ct}\,dt, \noz
\end{align*}
which implies that $\|f\|_{{{\rm {BMO}}}_{w}^p(\rn, \ch_{\fz}^{\dz})}\ls \|f\|_{{{\rm {BMO}}}^q_w(\rn, \ch_{\fz}^{\dz})}$.
Similarly, we also have $\|f\|_{{{\rm {BMO}}}_{w}^q(\rn, \ch_{\fz}^{\dz})}\ls \|f\|_{{{\rm {BMO}}}^p_w(\rn, \ch_{\fz}^{\dz})}$,
which implies \eqref{eq-0629-1}. This completes the proof of Theorem \ref{them-0629-1}.
\end{proof}

\begin{proof}[Proof of Theorem \ref{them-127-2}]
Via Proposition \ref{tm-127-1}, similar to the argument that used in the proof of Theorem \ref{them-0629-1},
we know that, for $w\in\ca_{p,\dz}$, the weighted $\rm{BLO}$ space ${{\rm {BLO}}}^q_{w}(\rn, \ch_{\fz}^{\dz})$ coincides with  ${\rm{BLO}}(\rn, \ch_{\fz}^{\dz})$ in the sense of equivalent quasi-norms; the details are omitted.
\end{proof}

\medskip

\noindent\textbf{Acknowledgements}
~The authors would like to sincerely thank Professor Daniel Spector for his valuable suggestions and for bringing reference \cite{bcr25} to their attention. L. Huang would like to thank Professor Zipeng Wang from Chongqing University for some helpful comments.
This project is partly supported
by the Science and Technology Projects of Guangzhou (SL2024A04J00209), the National Natural
Science Foundations of China (12201139), and the Natural Science Foundation of Hunan province (2024JJ3023).



\medskip

\noindent Long Huang

\medskip

\noindent School of Mathematics and Information Science,
Guangzhou University, Guangzhou, 510006, The People's Republic of China

\smallskip

\noindent {\it E-mail}: \texttt{longhuang@gzhu.edu.cn}

\medskip
\medskip

\noindent  Yangzhi Zhang and Ciqiang Zhuo (Corresponding author)
	
\smallskip
	
\noindent Key Laboratory of Computing and Stochastic Mathematics
(Ministry of Education), School of Mathematics and Statistics,
Hunan Normal University,
Changsha, Hunan 410081, The People's Republic of China
	
\smallskip

\noindent {\it E-mails}:
\texttt{yzzhang@hunnu.edu.cn}
	
\noindent\phantom{{\it E-mails:}}
\texttt{cqzhuo87@hunnu.edu.cn}

\end{document}